\documentclass[10pt,reqno]{amsart}
\usepackage{amsmath} 
\usepackage{amssymb}
\usepackage{dsfont}
\usepackage[dvips,draft,final]{graphics}
\usepackage{amssymb,amsmath}
 \usepackage{mathrsfs}
\usepackage[T1]{fontenc}
\usepackage{fancyhdr}
\usepackage{url}

\usepackage{color}
\usepackage{graphicx}

\def\squarebox#1{\hbox to #1{\hfill\vbox to #1{\vfill}}}

\newtheorem{Thm}{Theorem}[section]

\newtheorem{lem}{Lemma}[section]
\newtheorem{proposition}{Proposition}[section]

\numberwithin{equation}{section}

\newcommand{\bel}{\begin{equation} \label}
\newcommand{\ee}{\end{equation}}

\newcommand{\pd}{\partial}

\newcommand{\R}{\mathbb{R}}

\def\epsilon{\varepsilon}
\def\phi {\varphi}

\newtheorem{rem}{Remark}[section]
\newtheorem{prop}{Proposition}[section]

\providecommand{\abs}[1]{\left\lvert#1\right\rvert}
\providecommand{\norm}[1]{\left\lVert#1\right\rVert}
\numberwithin{equation}{section}

\renewcommand{\leq}{\leqslant}
\renewcommand{\geq}{\geqslant}
\providecommand{\abs}[1]{\left\lvert#1\right\rvert}
\providecommand{\norm}[1]{\left\lVert#1\right\rVert}

\def\beq{\begin{equation}}
\def\eeq{\end{equation}}
\newcommand{\bea}{\begin{eqnarray}}
\newcommand{\eea}{\end{eqnarray}}
\newcommand{\beas}{\begin{eqnarray*}}
\newcommand{\eeas}{\end{eqnarray*}}

{

\begin{document}

\title[Logarithmic stability in determining the zero order term]{Logarithmic stability in determining the time-dependent zero order coefficient in a parabolic equation from a partial Dirichlet-to-Neumann map. Application to the determination of a nonlinear term}

\author[Mourad Choulli]{Mourad Choulli}
\address{\dag IECL, UMR CNRS 7502, Universit\'e de Lorraine, Boulevard des Aiguillettes BP 70239 54506 Vandoeuvre Les Nancy cedex- Ile du Saulcy - 57 045 Metz Cedex 01 France}
\email{mourad.choulli@univ-lorraine.fr}

\author[Yavar Kian]{Yavar Kian}
\address{Aix Marseille Universit\'e, CNRS, CPT UMR 7332, 13288 Marseille, France \& Universit\'e de Toulon, CNRS, CPT UMR 7332, 83957 La Garde, France}
\email{yavar.kian@univ-amu.fr}

\begin{abstract}
We give a new stability estimate for the problem of determining  the time-dependent zero order coefficient in a parabolic equation from a partial parabolic Dirichlet-to-Neumann map. The novelty of our result is that, contrary to the previous works, we do not need any measurement on the final time. We also show how this result can be used to establish a stability estimate for the problem of determining the nonlinear term in a semilinear parabolic equation from the corresponding ``linearized'' Dirichlet-to-Neumann map. The key ingredient in our analysis is a parabolic version of an elliptic Carleman inequality due to Bukhgeim and Uhlmann \cite{BU}. This parabolic Carleman inequality enters in an essential way in the construction of CGO solutions that vanish at a part of the lateral boundary.


\end{abstract}

\maketitle

%%%%%%%%%%%%%%%%%%%%%%%%%%%%%%%%%%%%%%%%%%%%%%%%%%%%%%%

\section{Introduction}
\label{sec-intro}
\setcounter{equation}{0}

Let $\Omega$ be a $C^2$-bounded domain of $\mathbb{R}^n$, $n\ge 2$, with boundary $\Gamma$ and, for $T>0$, set
\[
Q=\Omega\times(0,T),\quad \Omega_+=\Omega \times \{0\},\quad \Sigma=\Gamma \times(0,T).
\] 

In all of this text, the symbol $\Delta$ denotes the Laplace operator with respect to the space variable $x$. 

\smallskip
Consider the initial boundary value problem, abbreviated to IBVP in the sequel,
\begin{equation}\label{eqq1}
\left\{\begin{array}{ll}(\partial_t-\Delta +q(x,t))u=0\quad \textrm{in}\ Q,
\\  u_{|\Omega _+}=0,%u_0,
\\ u_{|\Sigma}=g.\end{array}\right.\end{equation}

We are mainly interested in the stability issue of the problem of determining the time-dependent coefficient $q$ by measuring the corresponding solution $u$ of the IBVP \eqref{eqq1} on a part of $\Sigma$ when $g$ is varying in a suitable set of data, which means that we want to  establish a stability estimate of recovering $q$ from a partial Dirichlet-to-Neumann map, denoted by DtN map in the sequel.

\smallskip
The IBVP \eqref{eqq1} is for instance a typical model of the propagation of the heat through a time-evolving homogeneous body. The goal is to  determine the coefficient $q$, who contains some properties of the body, by applying a heat source on some part of the boundary of the body and  measuring the temperature on another part of the boundary of the body. Another classical inverse parabolic problem consists in determining the diffusion coefficient of an inhomogeneous medium through an IBVP for the equation $\partial_tv-\textrm{div}(a(t,x)\nabla v)= 0$. This last problem can be converted to the previous one  by means of the Liouville transform $u=\sqrt{a}v$. 
In many applications we are often lead to determine  physical quantities via parabolic IBVP's including nonlinear terms from boundary measurements. For instance such kind of problems appears in reservoir simulation, chemical kinetics and aerodynamics.

\smallskip
We introduce the functional space setting in order to define the DtN map associated to the IBVP \eqref{eqq1}.  Following  Lions and Magenes \cite{LM2}, $H^{-r,-s}(\Sigma )$, $r,s>0$, denotes the dual space of
\[
H^{r,s}_{,0}(\Sigma )=L^2(0,T;H^r(\Gamma ))\cap H_0^s(0,T;L^2(\Gamma )).
\]
From Proposition \ref{p2} in Section 2,  for $q\in L^\infty(Q)$ and $g\in H^{-{1\over 2},-{1\over4}}(\Sigma)$, the IBVP \eqref{eqq1} admits a unique transposition solution $u_{q,g}\in L^2(Q)$. Additionally, where $\nu$ is  the unit exterior normal vector field on $\Gamma$, the following parabolic DtN  map 
\begin{align*}
\Lambda_q:H^{-{1\over 2},-{1\over4}}(\Sigma)&\rightarrow H^{-{3\over 2},-{3\over4}}(\Sigma)
\\
g&\mapsto \pd_\nu u_{q,g}
\end{align*}
is bounded.

\smallskip
For $\omega\in\mathbb S^{n-1}$, set
\[
\Gamma_{\pm ,\omega}=\{x\in\Gamma ;\; \pm \nu(x)\cdot\omega > 0\}
\]
and $\Sigma_{\pm ,\omega}=\Gamma_{\pm,\omega}\times (0,T)$. 

\smallskip
Fix $\omega _0\in \mathbb{S}^{n-1}$, $\mathcal{U}_\pm$ a neighborhood of $\Gamma_{\pm ,\omega _0}$ in $\Gamma$ and set $\mathcal{V}_+ =\mathcal{U}_+ \times [0,T]$, $\mathcal{V}_- =\mathcal{U}_- \times (0,T)$. Define then the partial parabolic DtN operator
\begin{align*}
\widehat{\Lambda}_q:H^{-{1\over 2},-{1\over4}}(\Sigma)\cap \mathscr{E}'(\mathcal{V}_+)&\rightarrow H^{-{3\over 2},-{3\over4}}(\mathcal{V}_-)
\\
g&\mapsto \pd_\nu u_{q,g}{_{|\mathcal{V}_-}}.
\end{align*}
Here $\mathscr{E}'(\mathcal{V}_+)=\{ u\in \mathscr{E}'(\Gamma \times \mathbb{R});\; \mbox{supp}(u)\subset \mathcal{V}_+\}$ and $H^{-{3\over 2},-{3\over4}}(\mathcal{V}_-)$ denotes the quotient space
\[
H^{-{3\over 2},-{3\over4}}(\mathcal{V}_-)=\{ h=g_{|\mathcal{V}_-};\; g\in H^{-{3\over 2},-{3\over4}}(\Sigma)\}.
\]
Henceforth, the space $H^{-{3\over 2},-{3\over4}}(\mathcal{V}_-)$ is equipped with its natural quotient norm. 

\smallskip
We note that $\Lambda _q-\Lambda_{\widetilde{q}}$ has a better regularity than $\Lambda _q$ and $\Lambda_{\widetilde{q}}$ individually. Precisely, Proposition \ref{p3} in Section 2 shows that actually $\Lambda _q-\Lambda_{\widetilde{q}}\in \mathscr{B} (H^{-{1\over 2},-{1\over4}}(\Sigma),H^{{1\over 2},{1\over4}}(\Sigma ))$. The same remark is also valid for  $\widehat{\Lambda} _q-\widehat{\Lambda}_{\widetilde{q}}$. 

\smallskip
The first author establish in \cite{Ch} a logarithmic stability estimate for the problem of determining the zero order term from the parabolic DtN map $\Lambda _q$ together with the final data $g\rightarrow u_{q,g}(\cdot ,T)$. As a first result in the present work we improve this stability estimate.

\smallskip
In this text, the unit ball of a Banach space $X$ will be denoted in the sequel by $B_X$. 

\medskip
For $\frac{1}{2(n+3)}<s<\frac{1}{2(n+1)}$, set
\begin{equation}\label{Psi}
\Psi _s(\rho )=\rho + |\ln \rho |^{-\frac{1-2s(n+1)}{8}},\;\; \rho >0,
\end{equation}
extended by continuity at $\rho =0$ by setting $\Psi _s(0)=0$. 

\begin{Thm}\label{tt1} 
Fix $m>0$ and $\frac{1}{2(n+3)}<s<\frac{1}{2(n+1)}$. There exists a constant $C>0$, that can depend only on $m$, $Q$ and $s$, so that, for any  $q,\widetilde{q} \in mB_{L^\infty (Q)}$, 
\begin{equation}\label{tt1a} 
\norm{q_1-q_2}_{H^{-1}(Q)}\leq C\Psi _s\left(\norm{\Lambda_{q_1}-\Lambda_{q_2}}\right).
\end{equation}
Here $\norm{\Lambda_{q_1}-\Lambda_{q_2}}$ stands for the norm of $\Lambda_{q_1}-\Lambda_{q_2}$ in $\mathscr{B}(H^{-{1\over2},-{1\over4}}(\Sigma);H^{{1\over2},{1\over4}}(\Sigma))$. 
\end{Thm}

In the case of the infinite cylindrical domain $Q=\Omega \times (0,\infty )$, Isakov \cite{I4} got a stability estimate of determining $q=q(x)$ from the full parabolic DtN map by combining the decay in time of solutions of parabolic equations and the stability estimate in \cite{A} concerning the problem of determining the zero order coefficient in a elliptic BVP from a full DtN map.
For finite cylindrical domain $Q=\Omega \times (0,T)$, to our knowledge, even for time-independent coefficients, there is no result in the literature dealing with the stability issue of recovering of $q$ from the full DtN map $\Lambda_q$.

\smallskip
In fact Theorem \ref{tt1} is obtained as by-product of the analysis we developed to derive a logarithmic stability estimate for the problem of determining $q$ from the partial parabolic DtN map $\widehat{\Lambda}_q$.  This result is stated in the following theorem, where

\begin{equation}\label{Phi}
\Phi _s(\rho )=\rho + |\ln |\ln \rho ||^{-s},\;\; \rho >0,\; s>0,
\end{equation}
extended by continuity at $\rho =0$ by setting $\Phi _s(0)=0$.

\begin{Thm}\label{t1} 
Let $m>0$,  there exist two constants $C>0$ and $s\in (0,1/2)$,  that can depend only  on $m$, $Q$ and $\mathcal{V}_\pm$, so that, for any $q,\widetilde{q} \in mB_{L^\infty (Q)}$,
\begin{equation}\label{t1a} 
\norm{q_1-q_2}_{H^{-1}(Q)}\leq C \Phi _s\left(\|\widehat{\Lambda}_q-\widehat{\Lambda}_{\widetilde{q}}\| \right).
\end{equation}
Here $\|\widehat{\Lambda}_q-\widehat{\Lambda}_{\widetilde{q}}\|$ denotes the norm of $\widehat{\Lambda}_q-\widehat{\Lambda}_{\widetilde{q}}$ in $\mathscr{B}(H^{-{1\over 2},-{1\over4}}(\Sigma);H^{{1\over2},{1\over4}}(\mathcal{V}_-))$. 
\end{Thm}

It is worth mentioning that the uniqueness holds for the problem of determining $q$ from the partial DtN operator that maps the boundary condition $g$ supported on $\Gamma _0\times (0,T)$ into $\partial _\nu u_{q,g}$ restricted to $\Gamma _1\times (0,T)$, where $\Gamma _i$, $i=0,1$, are arbitrary nonempty open subsets of $\Gamma$. This result is stated as Theorem 3.27 in \cite[page 197]{Ch}. We note that the stability estimate corresponding to this uniqueness result remains an open problem.

\smallskip
In order to avoid the data at the final time, we adopt a strategy based on a parabolic Carleman inequality to construct the so-called CGO solutions vanishing at a part of the lateral boundary similar to that already used by the second author in \cite{Ki1,Ki2,Ki3} for determining  time-dependent coefficients in a wave equation.

\smallskip
There is a wide literature devoted to inverse parabolic problems and specifically the determination of time-dependent coefficients. We just present briefly some typical results. Canon and  Esteva \cite{CE86-1}
proved a logarithmic stability estimate for the determination of the  support of a source term in a one dimension parabolic equation from a boundary measurement.  This result was extended to three dimension heat equation in \cite{CE86-2}. The case of a non local measurement was considered by  Canon and  Lin  in \cite{CL88,CL90}. In \cite{Ch91-1}, the first author  proved existence, uniqueness and Lipschitz stability  for the determination of a time-dependent coefficient appearing in an abstract integro-differential equation, extending earlier results in \cite{Ch91-2}. The first author and  Yamamoto established in \cite{CY06} a stability estimate for the inverse problem of determining a source term appearing in a  heat equation from Neumann boundary measurements. In \cite{CY11}, the first author and  Yamamoto  considered an inverse semi-linear parabolic problem of  recovering the coefficient used  to reach a desired temperature along a curve. In \cite{I}, Isakov extended the construction of complex geometric optics solutions, introduced in \cite{SU}, to various PDE's including hyperbolic and parabolic equations to prove the density of products of solutions. One can get from the results in \cite{I} the unique determination of $q$ from the measurements on the lateral boundary  together with data at the final time.   
When the space domain is cylindrical, adopting the strategy introduced in \cite{BK}, the second author and Ga\" \i tan proved in \cite{GK} that the time-dependent zero order coefficient can be recovered uniquely from a single boundary measurement. Based on properties of fundamental solutions of parabolic equations, we proved in \cite{CK} Lipschitz stability of determining the time-dependent part of the zero order coefficient in a parabolic IBVP from a single boundary measurement. 

\smallskip
We also mention the recent works related to the determination of a time-dependent coefficients in IBVP's for hyperbolic, fractional diffusion and dynamical  Schr\"odinger equations \cite{Be,CKS1,CKS3,FK,Ki1,Ki2,Ki3}.

\smallskip
We  point out that, concerning the elliptic case, a stability estimate corresponding to the uniqueness result by Bukhgeim and Uhlmann \cite{BU} was established by Heck and Wang \cite{HW}. This result was improved by the authors and Soccorsi in \cite{CKS2}. Caro, Dos Santos Ferreira and Ruiz \cite{CDR} obtained recently a logarithmic stability estimate corresponding to the uniqueness result by Kenig, Sj\"ostrand and G. Uhlmann \cite{KSU}. Both the determination of the scalar potential and the conductivity in a periodic cylindrical domain from a partial DtN map was tackled in \cite{CKS4, CKS5}. We just quote these few references. But, of course, there is a tremendous literature on this subject in connection with the famous Calder\`on's problem.

\smallskip
Considering time-dependent unknown coefficients in parabolic equations is very useful when treating the determination of the nonlinear term appearing in a semilinear parabolic equation. We discuss this topic in Section 6. Uniqueness results for such kind of inverse semilinear parabolic problems was already established by Isakov \cite{I2,I3,I4}. Stability estimates and uniqueness in the case of a single boundary lateral measurement has been proved in \cite{COY,Kl} for a restricted class of unknown nonlinearities.

\smallskip
The rest of this text is organized as follows. Section 2 is devoted to existence and uniqueness of solutions of the IBVP \eqref{eqq1} in a weak sense. Following a well established terminology, we call these weak solutions the transposition solutions. We prove in Section 3  a Carleman inequality which is, as we said before, the key point in constructing CGO solutions that vanish on a part of the lateral boundary. These CGO solutions are constructed in Section 4. Theorems \ref{tt1} and \ref{t1} are proved in Section 5. We finally apply in section 6  the result in Theorem \ref{tt1} to the problem of determining the non linear term is a semilinear parabolic equation from the corresponding ``linearized'' DtN map.

%%%%%%%%%%%%%%%%%%%%%%%%%%%%%%%%%%%%%%%%%%%%%%%%%%%%%%%%%%%%%%%%%%%%%%%%%%%

\section{Transposition solutions}

This section is mainly dedicated to the IBVP \eqref{eqq1}. We construct the necessary framework leading to the rigorous definition of the parabolic DtN map.

\smallskip
For sake of simplicity, we limit our study in this section to real-valued functions. But all the results are extended without any difficulty to complex-valued functions.

\smallskip
Henceforth
\[
H_{\pm}=\{u\in L^2(Q);\; (\pm\partial_t-\Delta ) u\in L^2(Q)\}
\]
is equiped with its natural norm
\[
\norm{u}_{H_{\pm}}=\left(\norm{u}_{L^2(Q)}^2+\norm{(\pm\partial_t-\Delta ) u}_{L^2(Q)}^2\right)^{1/2}.
\]
Proceeding similarly to  \cite[Theorem 4]{Ki2} or \cite[Lemma 2.1]{CKS4}, we show that  $\mathcal C^\infty(\overline{Q})$ is dense in $H_{\pm}$. 

\smallskip
In the rest of this text, $\Omega _-=\Omega \times \{T\}$.

\smallskip
Denote by $\mathcal{N}$ the set of $(g_0,g_1,w_+,w_-)\in H^{{3\over 2},{3\over4}}(\Sigma)\oplus H^{{1\over 2},{1\over4}}(\Sigma)\oplus H^1(\Omega )\oplus H^1(\Omega )$ satisfying the compatibility conditions
\begin{equation}\label{p1c}
g_0(\cdot ,0)=w_+ \;\; \mbox{and}\;\; g_0(\cdot ,T)=w_-\;\; \mbox{on}\; \Gamma .
\end{equation}

\smallskip
From the results in \cite[Section 2.5, page 17]{LM2}, 
for any $(g_0,g_1,w_+,w_-)\in \mathcal{N}$, there exists \[w=E(g_0,g_1,w_+,w_-)\in H^{2,1}(Q)\] such that, in the trace sense, 
\begin{equation}\label{p1d}
w_{|\Sigma}=g_0,\quad \pd_\nu w_{|\Sigma}=g_1,\quad   w_{|\Omega_{\pm}}=w_\pm
\end{equation}
and
\begin{equation}\label{p1e} \norm{w}_{H^{2,1}(Q)}\leq C\left(\norm{g_0}_{H^{{3\over 2},{3\over4}}(\Sigma)}+\norm{g_1}_{H^{{1\over 2},{1\over4}}(\Sigma)}+\norm{w_-}_{H^1(\Omega)}+\norm{w_+}_{H^1(\Omega)}\right),
\end{equation}
for some constant $C>0$ depending only on $Q$. 

\smallskip
We recall that $H_0^{1\over4}(0,T;L^2(\Gamma ))$ coincides with $H^{1\over4}(0,T;L^2(\Gamma ))$. This fact is more or less known, but for sake of completeness we provide its proof in Lemma \ref{iii} of  Appendix A. Whence $H_{,0}^{{1\over 2},{1\over4}}(\Sigma)$ is identified to $H^{{1\over 2},{1\over4}}(\Sigma)$. Therefore, we identify in the sequel  the dual space of $H^{{1\over 2},{1\over4}}(\Sigma)$ to $H^{-{1\over 2},-{1\over4}}(\Sigma)$.

\begin{prop}\label{p1} The maps $\tau_{j}$, $j=0,1$, and $r_\pm$ defined for $v \in \mathcal C^\infty(\overline{Q})$ by
\begin{equation}\label{p1a} 
\tau_{0}v=v_{|\Sigma},\quad \tau_{1}v=\pd_\nu v_{|\Sigma},\quad r_\pm v=v_{|\Omega _\pm}
\end{equation}
are extended to bounded operators
\begin{equation}\label{p1b} \tau_{0}: H_{\pm}\rightarrow H^{-{1\over 2},-{1\over4}}(\Sigma),\quad \tau_{1}: H_{\pm}\rightarrow H^{-{3\over 2},-{3\over4}}(\Sigma),\quad r_-,\, r_+: H_\pm\rightarrow H^{-1}(\Omega).\end{equation}
\end{prop}

\begin{proof}  

\smallskip 
Assume that $w=E(0,g_1,0,0)$ for some $(0,g_1,0,0)\in \mathcal{N}$. Then Green's formula,  where $v\in \mathcal C^\infty(\overline{Q})$, yields
\[
\left\langle \tau_0 v,\tau_1 w\right\rangle=\int_Q(\pd_t-\Delta)vwdxdt-\int_Qv(-\pd_t-\Delta )wdxdt
\]
which, combined with \eqref{p1d} and \eqref{p1e}, entails
\[
\abs{\left\langle \tau_0 v,g_1\right\rangle}\leq C\norm{v}_{H_+}\norm{w}_{H^{2,1}(Q)}\leq C\norm{v}_{H_+}\norm{g_1}_{H^{{1\over 2},{1\over4}}(\Sigma)}.
\]
Whence $\tau_0$ can be extended by density to a bounded operator from $H_+$ to $H^{-{1\over 2},-{1\over4}}(\Sigma)$. 

\smallskip
We have similarly, by taking $w=E(g_0,0,0,0)$ for some $(g_0,0,0,0)\in \mathcal{N}$,
\[
-\left\langle \tau_1 v,g_0\right\rangle=\int_Q(\pd_t-\Delta )vwdxdt-\int_Qv(-\pd_t-\Delta )wdxdt.
\]
This formula, \eqref{p1d} and \eqref{p1e} imply
\[
\abs{\left\langle \tau_1 v,g_0\right\rangle}\leq C\norm{v}_{H_+}\norm{w}_{H^{2,1}(Q)}\leq C\norm{v}_{H_+}\norm{g_0}_{H_{,0}^{{3\over 2},{3\over4}}(\Sigma)}.
\]
Consequently, $\tau_1$ can be extended by density to a bounded operator from $H_+$ to $H^{-{3\over 2},-{3\over4}}(\Sigma)$.

\smallskip
The remaining part of the proof can be proved in a similar manner. We leave to the interested reader to write down the details. 
\end{proof}

Introduce 
\begin{align*}
&\mathcal H_\pm=\{\tau_0u;\;  u\in H_+\; \mbox{and}\; r_\pm u=0\},
\\
&S_\pm=\{u\in L^2(Q);\; (\pm \pd_t-\Delta )u=0\; \; \mbox{and}\; r_\pm u=0\}.
\end{align*}

\begin{prop}\label{p1'} 
We have $\tau_0S_\pm=\mathcal H_\pm= H^{-{1\over 2},-{1\over4}}(\Sigma)$ and, for any $u\in S_+$,
\begin{equation}\label{p1'a}
\norm{u}_{H_+}=\norm{u}_{L^2(Q)}\leq C\norm{\tau_0u}_{H^{-{1\over 2},-{1\over4}}(\Sigma)},
\end{equation}
for some constant $C>0$ depending only on $Q$.
\end{prop}

\begin{proof} We already know  that $\tau_0S_\pm\subset\mathcal H_\pm\subset H^{-{1\over 2},-{1\over4}}(\Sigma)$. Then we have only to prove the reverse inclusions. To do that,  fix $g\in H^{-{1\over 2},-{1\over4}}(\Sigma)$ and, for $F\in L^2(Q)$, consider the IBVP
\begin{align*}
\left\{ 
\begin{array}{ll} 
(-\partial_t-\Delta )w =  F \;\; \mbox{in}\; Q ,
\\ 
w_{|\Omega _-\cup \Sigma}= 0.
\end{array}
\right.
\end{align*}
According to \cite[Theorem 1.43]{Ch}, this  IBVP  admits a unique solution $w=w_F\in H^{2,1}(Q)$ so that
\begin{equation}\label{p1'b}
\norm{\tau_1w_F}_{H^{{1\over 2},{1\over4}}(\Sigma)}\leq C\norm{w_F}_{H^{2,1}(Q)}\leq C\norm{F}_{L^2(Q)}.
\end{equation}
Thus, the linear form on $L^2(Q)$ given by 
$$F\in L^2(Q)\mapsto-\left\langle g,\tau_1 w_F\right\rangle $$ 
is continuous. Therefore, according to Riesz's representation theorem, there exists $u\in L^2(Q)$ such that
$$
(u,F)=-\left\langle g,\tau_1 w_F\right\rangle .
$$
Here and henceforth $(\cdot ,\cdot)$ is the usual scalar product on $L^2(Q)$. 

\smallskip
Taking $F=(-\pd_t-\Delta )w$ for some $w\in \mathcal C^\infty_0(Q)$, we get $(\pd_t-\Delta) u=0$ on $Q$. On the other hand, the choice of $F=(-\pd_t-\Delta) w$, where $w=E(0,0,w_+,0)\in H^{2,1}(Q)$ with $(0,0,w_+,0)\in \mathcal{N}$, yields $r_+u=0$ and then $u\in S_+$. Finally,  $F=(-\pd_t-\Delta) w$, for some $w=E(0,g_1,0,0)\in H^{2,1}(Q)$ with $(0,g_1,0,0)\in \mathcal{N}$, gives $\tau_0u=g$. In other words, we proved that $g\in \tau_0S_+$ and consequently $\tau_0S_+=\mathcal H_\pm=H^{-{1\over 2},-{1\over4}}(\Sigma)$. We complete the proof by noting that \eqref{p1'a} follows from \eqref{p1'b} and the analysis we carried out for $S_+$ can be adapted with slight modifications to $S_-$.
\end{proof}

For $\epsilon =\pm $, consider the IBVP
\begin{equation}\label{eq1}
\left\{\begin{array}{ll}
(\epsilon \partial_t-\Delta +q(x,t))u=0\quad \textrm{in}\ Q,
\\
u_{|\Omega_\epsilon}=0,
\\  
u_{|\Sigma}=g .
\end{array}
\right.
\end{equation}

\begin{prop}\label{p2} 
For $g\in H^{-{1\over 2},-{1\over4}}(\Sigma)$ and  $q\in mB_{L^\infty(Q)}$, the IBVP \eqref{eq1} admits a unique transposition solution $u_{q,g}^\epsilon \in H_\epsilon$  satisfying 
\bel{p2a}
\norm{u_{q,g}^\epsilon}_{H_\epsilon}\leq C\norm{g}_{H^{-{1\over 2},-{1\over4}}(\Sigma)},
\ee
where the constant $C$ depends only on $Q$ and $m$. Additionally the parabolic DtN map \[ \Lambda_{q}: g\mapsto \tau_1u_{q,g}^+\] defines a bounded operator from $H^{-{1\over 2},-{1\over4}}(\Sigma)$ into
 $H^{-{3\over 2},-{3\over4}}(\Sigma)$.
\end{prop}

 \begin{proof} We give the proof for $\epsilon=+$. The case $\epsilon =-$ can be treated similarly. By Proposition \ref{p1'}, there exists a unique  $G\in S_+$ such that $\tau_0G=g$ and  
 $$
 \norm{G}_{L^2(Q)}\leq C\norm{g}_{H^{-{1\over 2},-{1\over4}}(\Sigma)}.
 $$ 
Consider the IBVP
\begin{align*}
\left\{ 
\begin{array}{ll} 
(\partial_t-\Delta +q)w= -qG\quad \mbox{in}\; Q ,
\\ 
w_{|\Omega_+\cup \Sigma }=0.
 \end{array}\right.
\end{align*}
As $-qG\in L^2(Q)$,   we get from \cite[Theorem 1.43]{Ch} that this  IBVP has a unique solution $w\in H^{2,1}(Q)$ satisfying
\begin{equation}\label{0.1}
\norm{w}_{H^{2,1}(Q)} \leq C\norm{qG}_{L^2(Q)}\leq C\norm{q}_{L^\infty(Q)}\norm{g}_{H^{-{1\over 2},-{1\over4}}(\Sigma)} .
\end{equation}
Hence, $u_{q,g}^+=w+G\in H_+$ is the unique transposition solution of \eqref{eq1} and \eqref{p2a} follows from \eqref{0.1}. 

\smallskip
Now, according to Proposition \ref{p1}, we have $\tau_1u_{q,g}^+\in H^{-{3\over 2},-{3\over4}}(\Sigma)$  with
\[
\begin{aligned}\norm{\tau_1u_{q,g}^+}^2_{H^{-{3\over 2},-{3\over4}}(\Sigma)}\leq C\norm{u_{q,g}^+}^2_{H_+}&=C\left(\norm{u_{q,g}^+}^2_{L^2(Q)}+\norm{qu_{q,g}^+}^2_{L^2(Q)}\right)\\
\ &\leq C\left(1+\norm{q}^2_{L^\infty(Q)}\right)\norm{u_{q,g}^+}_{L^2(Q)}^2,
\end{aligned}
\]
which, in combination with \eqref{p2a}, entails 
\[
\Lambda_q: H^{-{1\over 2},-{1\over4}}(\Sigma)\rightarrow H^{-{3\over 2},-{3\over4}}(\Sigma): g\mapsto \tau_1u_{q,g}^+
\]
defines a bounded operator.
\end{proof}

The identity in the following proposition will be very useful in our analysis.

\begin{prop}\label{p3} 
Fix $m>0$ and let $q, \widetilde{q}\in mB_{L^\infty(Q)}$. Then $\Lambda_q-\Lambda_{\widetilde{q}}$ is a bounded operator from $H^{-{1\over 2},-{1\over4}}(\Sigma)$ into  $H^{{1\over 2},{1\over4}}(\Sigma)$ and
\begin{equation}\label{p3a} 
\left\langle(\Lambda_{q}-\Lambda_{\widetilde{q}})g, h  \right\rangle =\int_Q(q-\widetilde{q})u_{q,g}^+u_{\widetilde{q},h}^-dxdt ,\;\; g,h\in  H^{-{1\over 2},-{1\over4}}(\Sigma).
\end{equation}
\end{prop}

\begin{proof} It is straightforward to check that $u=u_{\widetilde{q},g}^+-u_{q,g}^+$ is the solution of the IBVP
$$
\left\{ 
\begin{array}{ll} (\partial_t-\Delta +\widetilde{q})u=(q-\widetilde{q})u_{q,g}^+\;\;  \mbox{in}\; Q ,
\\ 
u_{|\Omega _+\cup \Sigma}=0.
 \end{array}
 \right.
 $$
Therefore, $u\in H^{2,1}(Q)$ and by the trace theorem \cite[Theorem 2.1, page 9]{LM2} $(\Lambda_q-\Lambda_{\widetilde{q}})g=\pd_\nu u\in H^{{1\over 2},{1\over4}}(\Sigma)$. Additionally \eqref{p2a} implies
\begin{align*}
\norm{(\Lambda_q-\Lambda_{\widetilde{q}})g}_{H^{{1\over 2},{1\over4}}(\Sigma)}=\norm{\pd_\nu u}_{H^{{1\over 2},{1\over4}}(\Sigma)}&\leq C\norm{u}_{H^{2,1}(Q)}\\&\leq C\norm{(q-\widetilde{q})u_{q,g}^+}_{L^2(Q)}\\ &\leq C\norm{g}_{H^{-{1\over 2},-{1\over4}}(\Sigma)},
\end{align*}
where $C$ is a generic constant depending only on $Q$ and $m$.

\smallskip
Let $h=\tau_0H$, where $H\in \mathcal C^\infty(\overline{Q})\cap S_-$. We find by making  integrations by parts

\begin{align}
-\int_\Sigma\pd_\nu uh\, d\sigma dt&=\int_Q(\pd_t-\Delta+\widetilde{q})uu_{\widetilde{q},h}^-dxdt-\int_Qu(-\pd_t-\Delta +\widetilde{q})u_{\widetilde{q},h}^-dxdt\label{p3b}
\\
&=\int_Q(q-\widetilde{q})u_{q,g}^+u_{\widetilde{q},h}^-dxdt. \nonumber
\end{align}

As $\mathscr{S}_-=C^\infty(\overline{Q})\cap S_-$  is dense in $S_-$, $\tau_0\mathscr{S}_-$ is dense in $H^{-{1\over 2},-{1\over4}}(\Sigma)$ according to Proposition \ref{p1'}. Whence \eqref{p3a} is deduced from \eqref{p3b} by density. 
\end{proof}

%%%%%%%%%%%%%%%%%%%%%%%%%%%%%%%%%%%%%%%%%%%%%%%%%%%%%%%%%%%%%%%

\section{A Carleman inequalities}

We establish in this section a parabolic version of an elliptic Carleman inequality due to Bukgheim and Uhlmann \cite{BU}. This inequality is used in an essential way in constructing CGO solutions vanishing at a part of the lateral boundary.

\smallskip
In this section $\epsilon =\pm$ and, for $\omega \in \mathbb{S}^{n-1}$ and $\rho\ge 0$, 
\[
\varphi_{\epsilon, \omega ,\rho} (x,t)=e^{- \epsilon 2(\rho\omega\cdot x+\rho^2 t)},\;\; (x,t)\in \overline{Q},
\]
and $\psi_{\epsilon, \omega ,\rho}=\sqrt{\varphi_{\epsilon, \omega ,\rho}}$. 

\smallskip
Observe that $\psi_{\epsilon, \omega ,\rho}$ satisfies
\[
(\epsilon \partial _t-\Delta )\psi_{\epsilon, \omega ,\rho}=0,\;\; (x,t)\in \overline{Q}.
\]

\begin{Thm}\label{t2}  
There exists a constant $C>0$ depending only on $Q$ with the property that, for any $m>0$, we find $\rho_0>0$, depending only on $Q$ and $m$ so that, for any $q\in mB_{L^\infty(Q)}$, $\rho \ge \rho_0$ and $u\in C^2(\overline{Q})$ satisfying $u=0$ on $\Sigma \cup \Omega _\epsilon$,

\begin{align}
\int_{\Omega_{-\epsilon}} \varphi_{\epsilon ,\omega ,\rho}\abs{u}^2dx&+\rho\int_{\Sigma_{\epsilon,\omega}}\varphi_{\epsilon , \omega ,\rho}\abs{\partial_\nu u}^2\abs{\omega\cdot\nu } d\sigma dt+\rho^2\int_Q\varphi_{\epsilon ,\omega ,\rho}\abs{u}^2dxdt \label{t2b}
\\
&\le C\left(\int_Q\varphi_{\epsilon ,\omega ,\rho}\abs{(\epsilon\partial_t-\Delta+q)u}^2dxdt+\rho\int_{\Sigma_{-\epsilon,\omega}}\varphi_{\epsilon ,\omega ,\rho}\abs{\partial_\nu u}^2\abs{\omega\cdot\nu }d\sigma dt\right). \nonumber
\end{align}
\end{Thm}

\begin{proof}
It is enough to give the proof in the case of real-valued functions. We consider the case $\epsilon=+$. The case $\epsilon =-$ can be treated similarly.
Let $u\in C^2(\overline{Q})$ satisfying $u=0$ on $\Sigma \cup \Omega _+$ and set $v=\psi _{+,\omega ,\rho} u$. Straightforward computations give
\begin{equation}\label{c1c}
\psi_{+,\omega ,\rho}(\pd_t-\Delta ) u=P_{\omega ,\rho}v,\end{equation}
where 
$$
P_{\omega ,\rho}= \pd_t-\Delta - 2\rho\omega\cdot\nabla .
$$

In this proof, the symbol $\nabla$ denotes the gradient with respect to the variable $x$.

\smallskip
We split $P_{\omega ,\rho}$  into two terms 
\[
P_{\omega ,\rho}=-\Delta+ Q_{\omega ,\rho}\;\;  \mbox{with}\;\; Q_{\omega ,\rho}=\pd_t-2\rho\omega\cdot\nabla.
\]
 Hence
\[
\norm{P_{+,\rho,\omega}v}_{L^2(Q)}^2\geq \norm{Q_{\omega ,\rho}v}^2-2\left(\Delta v,Q_{\omega ,\rho}v\right).
\]
But
\[
-2\left(\Delta v,Q_{\omega ,\rho}v\right)_{L^2(Q)}=-2\int_Q\Delta v\pd_tvdxdt+2\rho\int_Q\Delta v\omega\cdot\nabla vdxdt.
\]
Applying Green's formula, we get
\[
-2\int_Q\Delta v\pd_tvdxdt=\int_Q\pd_t|\nabla v|^2dxdt=\int_{\Omega_-} |\nabla v|^2dx .
\]
On the other hand, we have from the proof of \cite[Lemma 2.1]{BU}
\[
2\int_Q\Delta v\omega\cdot\nabla vdxdt=\int_\Sigma |\pd_\nu v|^2\omega\cdot\nu d\sigma dt.
\]
Combining these formulas, we obtain
\begin{equation}\label{0.2}
\norm{P_{\omega ,\rho}v}_{L^2(Q)}^2\geq \int_{\Omega_-} |\nabla_x v|^2+\rho\int_\Sigma |\pd_\nu v|^2\omega\cdot\nu d\sigma dt+\norm{Q_{\omega ,\rho}v}^2_{L^2(Q)}.
\end{equation}

We need the following Poincar\'e type inequality to pursue the proof. This inequality is proved later in this text.

\begin{lem}
\label{l2} There exists a constant $C$, that can depend only on $\Omega$ so that, for any $\rho >2$ and $v\in  H^1(Q)$ satisfying $v=0$ on $\Sigma \cup \Omega_\epsilon$,
 \begin{equation}\label{l2b}
 \rho\norm{v}_{L^2(Q)}\leq C\norm{Q_{\omega ,\rho}v}_{L^2(Q)}.
 \end{equation}
 \end{lem}
 
Inequality \eqref{l2b} in \eqref{0.2} yields
\[ 
\int_{\Omega_-} | v|^2dx+\rho\int_{\Sigma_+} |\pd_\nu v|^2\omega\cdot\nu d\sigma dt+\rho^2\norm{v}^2_{L^2(Q)}\leq C_0\left(\norm{P_{\omega ,\rho}v}_{L^2(Q)}^2+\rho\int_{\Sigma_-} |\pd_\nu v|^2|\omega\cdot\nu| d\sigma dt\right).
\]
Here the constant $C_0$ depends only on $Q$. This gives \eqref{t2b} when $q=0$. For an arbitrary $q\in mB_{L^\infty (Q)}$, we have 
 \[
 \abs{\partial_tu-\Delta u }^2=\abs{\partial_tu-\Delta u+qu-qu}^2\le 2\abs{(\partial_t-\Delta+q)u}^2+2m^2\abs{u}^2.
 \]
 Fix $\rho_0>2\max(\sqrt{C_0}m,1)$. Then \eqref{t2b} follows with $C=2C_0$.
 \end{proof}
%Now that Theorem \ref{t2} is proved, we complete the proof of Lemma \ref{l2}.

\begin{proof}[Proof of Lemma \ref{l2}] 
We consider the case $\epsilon =+$. The case $\epsilon =-$ is proved similarly.

\smallskip
Let $v\in  H^1(Q)$ satisfying $v=0$ on $\Sigma \cup \Omega_+$. A classical reflexion argument in $t$ with respect to $T$ shows that $v$ is the restriction to $Q$ of a function belonging to $H_0^1(\Omega \times (0,2T))$. Therefore, by density, it is enough to give the proof when $v\in C_0^\infty (\Omega \times (0,2T))$ that we consider in the sequel as a subset of $C^\infty_0(\Omega \times (0,+\infty ))$.

\smallskip
Let $v\in C^\infty_0(\Omega \times (0,+\infty ))$. If $\eta=(\eta_x ,\eta_t)\in\R^n \times\R$ is defined by
\[ \eta_x={-2\rho\omega \over \sqrt{1+4\rho^2}}, \quad \eta_t={1\over \sqrt{1+4\rho^2}},\]
then
\[
v(x,t)=\int_{-\infty}^0\pd_s v((x,t)+s\eta)ds,\quad (x,t)\in Q.
\]
Fix $R>0$ such that $\Omega\subset \{|x|\leq R;\; x\in\R^n\}$. For $s<-4R$, $x\in\Omega$ and $\rho>2$, we get
 \[|x+s\eta_x|\geq {-s\over2}-|x|>R.\]
Thus, for $s<-4R$ and $x\in\Omega$, $x+s\eta_x\notin \Omega$ yielding $ v((x,t)+s\eta)=0$. Therefore
\[|v(x,t)|^2\leq\abs{\int_{-4R}^0\pd_s v((x,t)+s\eta)ds}^2,\quad (x,t)\in Q.\]
Applying Cauchy-Schwarz's inequality, we obtain
\begin{align*}
\int_Q|v(x,t)|^2dxdt\le 4R\int_Q\int_{-4R}^0|\pd_s v((x,t)&+s\eta)|^2dsdxdt
\\
&={4R\over 1+4\rho^2}\int_Q\int_{-4R}^0|Q_{\omega ,\rho} v((x,t)+s\eta)|^2dsdxdt.
\end{align*}
We make the substitution $\tau=t+s\eta_t$, $y=x+s\eta_x$ and  we apply Fubini's theorem. We find
\begin{align*}
\int_Q|v(x,t)|^2dxdt &\leq{4R\over 1+4\rho^2}\int_{-4R}^0\int_{-\infty}^{T+s\eta_t}\int_{\R^n}|Q_{\omega ,\rho} v(y,\tau )|^2dyd\tau ds
\\
& \leq{16R^2\over 1+4\rho^2}\int_Q|Q_{\omega ,\rho}v(x,t)|^2dxdt.
\end{align*}
Here we used both the fact that $\eta_t>0$ and $\mbox{supp}(v)\subset (0,+\infty)\times\Omega$.
\end{proof}

%%%%%%%%%%%%%%%%%%%%%%%%%%%%%%%%%%%%%%%%%%%%%%%%%

\section{CGO solutions vanishing on a part of the lateral boundary}

With the help of the Carleman inequality in the previous section, we construct in this section CGO solutions vanishing at a part of the lateral boundary.

\smallskip
As in the preceding section, for $\omega \in \mathbb{S}^{n-1}$, $\epsilon =\pm$ and $\rho\ge 0$, 
\[
\varphi_{\epsilon ,\omega ,\rho} (x,t)=e^{- \epsilon 2(\rho\omega\cdot x + \rho^2 t)},\;\; (x,t)\in \overline{Q},
\]
and $\psi_{\epsilon ,\omega ,\rho} =\sqrt{\varphi_{\epsilon ,\omega ,\rho}}$. Set, where $\delta >0$ and $\omega \in \mathbb{S}^{n-1}$,
\[
\Gamma_{+,\omega ,\delta}=\{x\in\Gamma;\;  \nu(x)\cdot \omega >\delta\}
\]
and $\Sigma_{+,\omega ,\delta}=\Gamma_{+,\omega ,\delta }\times (0,T)$. 

\smallskip
Fix $\omega\in \mathbb S^{n-1}$,  $\xi\in \R^n$ with $\xi\cdot\omega=0$, $\tau\in\R$ and $\rho\ge \rho_0$, $\rho_0$ is as in Theorem \ref{t2}. Let $\zeta =(\xi ,\tau)$ and
\begin{align*}
&\theta_+(x,t)=\left(1-e^{-\rho^{{3\over4}}t}\right)e^{-i(x,t)\cdot \zeta}
\\
&\theta _-(x,t)= \left(1-e^{-\rho^{{3\over4}}(T-t)}\right).
\end{align*}

\begin{Thm}\label{t3}
Let $m>0$. There exists a constant $C>0$ depending only on $Q$, $m$ and $\delta$ so that, for any $q\in mB_{L^\infty (\Omega )}$, the equation
\[
(\pm \partial _t -\Delta +q)u=0\;\; \mbox{in}\; Q
\]
has a solution $u_{\pm ,q}\in \mathcal{H}_\pm$, satisfying $u_{\pm ,q}=0$ on  $\Sigma_{+,\mp\omega ,\delta}$, of the form
\begin{equation}\label{GO1}
u_{\pm ,q}=\psi_{\mp ,\omega ,\rho} \left( \theta _\pm +w_{\pm ,q}\right),
\end{equation}
where $w_{\pm ,q}\in H_{\pm}$ is such that
\begin{equation}\label{GO4}
\norm{w_{+ ,q}}_{L^2(Q)}\leq C(\rho^{-{1\over4}}+\rho^{-1}\langle \zeta \rangle^2),\quad \norm{w_{- ,q}}_{L^2(Q)}\leq C\rho^{-{1\over4}}.
\end{equation}
Here and henceforth $\langle \zeta \rangle=\sqrt{1+|\zeta |^2}$.
\end{Thm}

In the rest of this section, for sake of simplicity, we use $\varphi_\epsilon$ and $\psi _\epsilon$ instead of $\varphi_{\epsilon ,\omega ,\rho}$ and $\psi_{\epsilon ,\omega ,\rho}$.

\smallskip
Before proving Theorem \ref{t3}, we establish some preliminary results. We firstly rewrite the Carleman inequality in Theorem \ref{t2} as an  energy inequality in weighted $L^2$-spaces. To this end, denote by $\|\cdot \|_{Q,\varphi_\epsilon }$, $\|\cdot \|_{\Omega _{-\epsilon} ,\varphi_\epsilon}$ and $\|\cdot \|_{\Sigma _{\pm \epsilon},\varphi_\epsilon\gamma }$, where $\gamma =|\nu \cdot \omega |$, the respective $L^2$-norms of $L^2(Q,\varphi_\epsilon dxdt)$,  $L^2(\Omega _{-\epsilon}, \varphi_\epsilon dx)$ and $L^2(\Sigma _{\pm \epsilon ,\omega}, \varphi_\epsilon \gamma d\sigma dt)$. Under these notations, inequality \eqref{t2b} takes the form

\begin{equation}\label{ca+}
\norm{u}_{\Omega _{-\epsilon},\varphi_\epsilon}+\rho^{\frac{1}{2}}\norm{\partial _\nu u}_{\Sigma _{\epsilon} ,\varphi_\epsilon\gamma}+\rho \norm{u}_{Q,\varphi_\epsilon}\leq C\left(\norm{(\epsilon\partial_t-\Delta +q)u}_{Q,\varphi_\epsilon}+\norm{\partial_\nu u}_{\Sigma _{-\epsilon},\varphi_\epsilon\gamma }\right),
\end{equation}
for any $\rho\ge \rho_0$ and $u\in \mathcal D_\epsilon =\{v\in C^2(\overline{Q});\; \ v_{| \Sigma\cup \Omega _\epsilon}=0\}$.

\smallskip
Again, for sake of simplicity, we use in the sequel the following notations 
\[
L^2_{Q,\varphi_\epsilon}=L^2(Q,\varphi_\epsilon dxdt)\quad L^2_{\Omega _{-\epsilon} ,\varphi_\epsilon}=L^2(\Omega _{-\epsilon}, \varphi_\epsilon dx),\quad L^2_{\Sigma _{\pm \epsilon} ,\varphi_\epsilon\gamma}=L^2(\Sigma _{\pm \epsilon ,\omega}, \varphi_\epsilon\gamma d\sigma dt).
\]
We identify the dual space of $L^2_{Q,\varphi_\epsilon}$ (resp. $L^2_{\Sigma _{- \epsilon} ,\varphi_\epsilon\gamma}$) by $L^2_{Q,\varphi_\epsilon^{-1}}$ (resp. $L^2_{\Sigma _{-\epsilon} ,\varphi_\epsilon^{-1}\gamma^{-1}}$). The respective norms of $L^2_{Q,\varphi_\epsilon^{-1}}$ and $L^2_{\Sigma _{- \epsilon} ,\varphi_\epsilon^{-1}\gamma^{-1}}$ are denoted respectively by $\|\cdot\|_{Q,\varphi_\epsilon^{-1}}$ and $\|\cdot \|_{\Sigma _{- \epsilon},\varphi_\epsilon ^{-1}\gamma^{-1}}$.

\smallskip
Consider the following subspace of $L^2_{Q,\varphi_\epsilon}\oplus L^2_{\Sigma _{-\epsilon} ,\varphi_\epsilon \gamma }$
\[
\mathcal M_\epsilon =\{((\epsilon \partial_t-\Delta +q)v,\partial_\nu v_{|\Sigma_{-\epsilon ,\omega}});\; v\in\mathcal D_\epsilon\}.
\]

\begin{lem}\label{l5}  
Fix $m>0$ and $q\in mB_{L^\infty (Q )}$.  Assume that $\rho\geq \rho_0$, with $\rho_0$ the constant of Theorem \ref{t2}, and let $(G,h)\in L^2_{Q,\varphi_\epsilon^{-1}}\oplus L^2_{\Sigma _{\epsilon} ,\varphi_\epsilon^{-1}\gamma^{-1}}$. Then there exists  $z\in L^2_{Q,\varphi_\epsilon}$  satisfying
\begin{align*}
&(-\epsilon \partial_t-\Delta +q)z=G\;\; \mbox{in}\; Q,
\\
&{z}_{|\Sigma_{\epsilon ,\omega}}=h,
\\ 
&z_{|\Omega _{-\epsilon}}=0,
\\
&\norm{z}_{Q,\varphi_\epsilon }\leq  C\left(\rho^{-1}\norm{G}_{Q,\varphi_\epsilon^{-1} }+\rho^{-\frac{1}{2}}\norm{h}_{\Sigma _{-\epsilon},\varphi_\epsilon^{-1}\gamma^{-1}}\right),
\end{align*} 
for some constant $C>0$ depending only on $Q$ and $m$.
\end{lem}

\begin{proof} 
Define on $\mathcal M_{\epsilon}$ the map $\mathscr{S}$ by
\[
\mathscr{S}[((\epsilon \partial_t-\Delta +q) f,\partial_\nu f_{\vert\Sigma_{-\epsilon ,\omega}})]=\left\langle f,G \right\rangle_{L^2(Q)}-\left\langle \partial_\nu f,h \right\rangle_{L^2(\Sigma_{\epsilon,\omega})},\quad f\in\mathcal D_\epsilon .
\]
In light of \eqref{ca+} we get, for $f\in\mathcal D_{\epsilon}$,
\begin{align*}
|\mathscr{S}[((\epsilon \partial_t-\Delta +q) f,&\partial_\nu f_{\vert\Sigma_{\epsilon ,\omega}})]|\le \|f\|_{Q,\varphi_\epsilon }\|G\|_{Q,\varphi_\epsilon^{-1} }+\|\partial _\nu f\|_{\Sigma _{\epsilon},\varphi _\epsilon\gamma }\|h\|_{\Sigma _{\epsilon},\varphi_\epsilon^{-1}\gamma^{-1}}
\\
&\le \left(\rho^{-1} \|G\|_{Q,\varphi_\epsilon^{-1} }+\rho^{-\frac{1}{2}}\|h\|_{\Sigma _{\epsilon},\varphi_\epsilon^{-1}\gamma^{-1}}\right)\left(\rho \|f\|_{Q,\varphi_\epsilon }+\|\partial _\nu f\|_{\Sigma _{\epsilon},\varphi _\epsilon\gamma }\right)
\\
&\le C\left(\rho^{-1} \|G\|_{Q,\varphi_\epsilon^{-1} }+\rho^{-\frac{1}{2}}\|h\|_{\Sigma _{\epsilon},\varphi_\epsilon^{-1}\gamma^{-1}}\right)\left(\norm{(\epsilon\partial_t-\Delta +q)u}_{Q,\varphi_\epsilon}+\norm{\partial_\nu u}_{\Sigma _{-\epsilon},\varphi_{\epsilon}\gamma }\right)
\\
&\le C\left(\rho^{-1} \|G\|_{Q,\varphi_\epsilon^{-1} }+\rho^{-\frac{1}{2}}\|h\|_{\Sigma _{\epsilon},\varphi_\epsilon^{-1}\gamma^{-1}}\right) \|((\epsilon \partial_t-\Delta +q) f, \partial_\nu f_{\vert\Sigma_{-\epsilon ,\omega}})\|_{L^2_{Q,\varphi_\epsilon}\oplus L^2_{\Sigma _{-\epsilon} ,\varphi_\epsilon \gamma }},
\end{align*}
where $C$ is the constant in \eqref{ca+}.  By Hahn Banach's extension theorem, $\mathscr{S}$ is extended to a continuous linear form on $L^2_{Q,\varphi_\epsilon}\oplus L^2_{\Sigma _{-\epsilon} ,\varphi_\epsilon \gamma }$. This extension is denoted again by $\mathscr{S}$. Additionally
\bel{l5a}
\norm{\mathscr{S}}\le C\left(\rho^{-1} \|G\|_{Q,\varphi_\epsilon^{-1} }+\rho^{-\frac{1}{2}}\|h\|_{\Sigma _{\epsilon},\varphi_\epsilon^{-1}\gamma^{-1}}\right),
\ee
where $\|\mathscr{S}\|$ denotes the norm of $\mathscr{S}$ in $\left[L^2_{Q,\varphi_\epsilon}\oplus L^2_{\Sigma _{-\epsilon} ,\varphi_\epsilon \gamma }\right]'$. Therefore, by Riesz's representation theorem, there exists $(z ,g)\in L^2_{Q,\varphi_\epsilon}\oplus L^2_{\Sigma _{-\epsilon} ,\varphi_\epsilon \gamma }$ so that, for any $f\in\mathcal D_\epsilon$,
\[
\mathscr{S}[((\epsilon \partial_t-\Delta +q) f,\partial_\nu f_{\vert\Sigma_{-\epsilon ,\omega}})]=\left( (\epsilon \partial_t-\Delta +q) f,z \right)_{L^2(Q)}+\left\langle\partial_\nu f,g \right\rangle_{L^2(\Sigma_{-\epsilon ,\omega})}.
\]
In other words we proved that, for any $f\in\mathcal D_\epsilon$, 
\bel{l5b}
\left( (\epsilon \partial_t-\Delta +q) f,z \right)_{L^2(Q)}+\left\langle\partial_\nu f,g \right\rangle_{L^2(\Sigma_{-\epsilon ,\omega})}=\left( f,G \right)_{L^2(Q)}-\left\langle \partial_\nu f,h \right\rangle_{L^2(\Sigma_{\epsilon ,\omega})}.
\ee

Taking $f\in C^\infty_0(Q)$, we get $(-\epsilon \partial_t-\Delta +q) z=G$ in $Q$. Whence $z\in H_{-\epsilon}$. In light of the trace theorem of Proposition \ref{p1}, the other properties of $z$ can be proved in a straightforward manner.
\end{proof}

\begin{proof}[Proof of Theorem \ref{t3}]
We give only the existence of $u_{+,q}$. That of $u_{-,q}$ can be established following the same method and therefore we omit the proof in this case. In the sequel $\rho\ge \rho_0$. Recall that we seek $u_{+,q}$ in the form $u_{+,q}=\psi_-(\theta_++w_{+,q})$ which means that $w_{+,q}$ must be the solution of the IBVP
\begin{equation}\label{t3b}
\left\{
\begin{array}{ll}
(\pd_t-\Delta +q)(\psi_-w)=-(\pd_t-\Delta +q)(\psi_-\theta_+)\;\; \mbox{in}\; Q,
\\
w_{|\Omega _+}=0,
\\
w_{|\Sigma_{+,-\omega ,\delta}}=-\theta_+ .
\end{array}
\right.
\end{equation}

The following identity is used in the sequel
\[
-(\pd_t-\Delta +q)(\psi_-\theta_+)=-\psi_-e^{-i(x,t)\cdot\zeta}\left[(-i\tau+|\xi|^2+q_1)\left(1-e^{-\rho^{{3\over4}}t}\right) -\rho^{{3\over4}}e^{-\rho^{{3\over4}}t}\right].
\]

Pick $\phi\in C^\infty_0(\R^n)$ so that   $\mbox{supp}(\phi )\cap\Gamma \subset \{x\in\Gamma ;\; \omega\cdot\nu(x)<-2\delta /3\}$ and $\phi=1$ on $\{x\in\Gamma ;\;\omega\cdot\nu(x)<-\delta \}=\Gamma_{+,-\omega ,\delta}$. Let
\begin{align*}
&h=-\varphi \theta _+\;\; \mbox{on}\; \Sigma_{-,\omega}=\Sigma_{+,-\omega},
\\
&G=-\psi_-e^{-i(x,t)\cdot\zeta}\left[(-i\tau+|\xi|^2+q_1)\left(1-e^{-\rho^{{3\over4}}t}\right) -\rho^{{3\over4}}e^{-\rho^{{3\over4}}t}\right].
\end{align*}

 From Lemma \ref{l5} (with $\epsilon =-$), there exists $z\in H_+$ satisfying
\[
\left\{
\begin{array}{ll}
(\partial_t-\Delta +q) z=G\;\;  \mbox{in}\; Q,
\\
z_{|\Omega _+}=0,
\\
z_{|\Sigma_{-,\omega}}=h.
\end{array}
\right.
\]
As $\psi_+=\psi_-^{-1}$, we see that $w_{+,q}=\psi_+z$ satisfies \eqref{t3b}. 

\smallskip
We complete the proof by showing that inequality \eqref{GO4} holds for $w_{+,q}$. To do that, we firstly note that
$$
\norm{G}_{Q,\varphi_+}\leq C(|\tau|+|\xi|^2+\rho^{{3\over4}})\leq C(\left\langle \zeta \right\rangle^2+\rho^{{3\over4}}),
$$
for some constant $C>0$ depending only on $m$ and $Q$. Whence
\[
\begin{aligned}\norm{w_{+,q}}_{L^2(Q)}=\norm{z}_{Q,\varphi_+}&\leq C\left(\rho^{-1}\norm{G}_{Q,\varphi_-^{-1}}+\rho^{-\frac{1}{2}}\norm{h}_{\Sigma_{-},\varphi_-^{-1}\gamma^{-1}}\right)\\
\ &\leq C\left(\rho^{-1}\norm{G}_{Q,\varphi_+}+\rho^{-\frac{1}{2}}\norm{h}_{\Sigma_{-},\varphi_+\gamma^{-1}}\right)\\
\ &\leq C\left(\rho^{-1}(\left\langle \zeta \right\rangle^2+\rho^{{3\over4}})+\rho^{-\frac{1}{2}}\norm{\varphi\gamma^{-{1\over2}}}_{L^2(\Sigma_{-,\omega})}\right)\\
\ &\leq C(\rho^{-{1\over4}}+\rho^{-1}\left\langle \zeta\right\rangle^2),\end{aligned}
\]
as it is expected.
\end{proof}

%%%%%%%%%%%%%%%%%%%%%%%%%%%%%%%%%%%%%%%%%%%%%%%%%%%%%%%%

\section{Proof of Theorems \ref{tt1} and \ref{t1}}

Fix $q,\widetilde{q}\in mB_{L^\infty (Q)}$ and set $p=(q-\widetilde{q})\chi_Q$, where $\chi_Q$ denotes the characteristic function of $Q$. For $\delta \in (0,1)$, let $\chi_{\pm,\omega ,\delta}\in C^\infty_0(\R^n)$ satisfying 
\[
\mbox{supp}(\chi_{\pm,\omega ,\delta })\cap\Gamma \subset \Gamma_{-,\mp\omega ,2\delta}\quad \mbox{and}\quad \chi_{\pm,\omega ,\delta}=1\;\;  \mbox{on}\;\; \Gamma_{-,\mp\omega ,\delta}.
\]
As usual, the operator $\chi_{-,\omega ,\delta}(\Lambda_{q}-\Lambda_{\widetilde{q}})\chi_{+,\omega ,\delta}$ acts as follows
\[\chi_{-,\omega ,\delta}(\Lambda_{q}-\Lambda_{\widetilde{q}})\chi_{+,\omega ,\delta}(g)=\chi_{-,\omega ,\delta}(\Lambda_{q}-\Lambda_{\widetilde{q}})(\chi_{+,\omega ,\delta}g),\;\;g\in H^{-{1\over2},-{1\over4}}(\Sigma).\]
Recall that the Fourier transform $\widehat{p}$ of $p$ is given by 
\[
\widehat{p}(\zeta )=(2\pi)^{-(n+1)/2}\int_{\R^{n+1}}p(x,t)e^{-i\zeta \cdot (x,t)}dxdt.
\]
We start with the preliminary lemma

\begin{lem}\label{l7} Let $\delta\in(0,1)$, $\omega\in\mathbb S^{n-1}$, $\zeta=(\xi ,\tau)\in\R^n\times\R$ with $\xi\cdot\omega=0$ and $\rho_0$ be as in Theorem \ref{t2}. Then
\bel{l7a}
\abs{\widehat{p}(\zeta )}\leq C\left(\rho^{-{1\over4}}+\rho^{-1}R^{2}+\norm{\chi_{-,\omega ,\delta}(\Lambda_{q}-\Lambda_{\widetilde{q}})\chi_{+,\omega ,\delta}}e^{c\rho^2}\right),\;\; |\zeta |\leq R ,\; \rho\ge\rho_0.
\ee
for some constants $c>0$ and $C>0$ depending only on  $m$, $\delta$ and $Q$. 

\smallskip
\noindent
Here $\norm{\chi_{-,\omega ,\delta}(\Lambda_{q}-\Lambda_{\widetilde{q}})\chi_{+,\omega ,\delta}}$ denotes the norm $\chi_{-,\omega ,\delta}(\Lambda_{q}-\Lambda_{\widetilde{q}})\chi_{+,\omega ,\delta}$ in $\mathscr{B}(H^{-{1\over2},-{1\over4}}(\Sigma),H^{{1\over2},{1\over4}}(\Sigma))$.
\end{lem}

\begin{proof} 
Let $g=\tau_0u_{+,q}$ (resp. $h=u_{-,\widetilde{q}}$), with $u_{+,q}$ (resp. $u_{-,\widetilde{q}}$) as in Theorem \ref{t3}. Then formula \eqref{p3a} yields
 \bel{l7b}
 \abs{\int_{\R^{n+1}}p(x,t)e^{-i\zeta\cdot (x,t)}dxdt}\leq \abs{\int_QZ(t,x)dxdt}+\abs{\left\langle \chi_{-,\omega ,\delta}(\Lambda_{q}-\Lambda_{\widetilde{q}})\chi_{+,\omega ,\delta}g,h\right\rangle}
 \ee
with
\[
Z=e^{-i\zeta \cdot (x,t)}w_{-,\widetilde{q}}+w_{+,q}+w_{+,q}w_{-,\widetilde{q}}+e^{-\rho^{3\over4}(T-t)}(e^{-i\zeta \cdot (x,t)}+w_{+,q})+e^{-\rho^{3\over4}t}(1+w_{-,\widetilde{q}})+e^{-\rho^{3\over4}T}.
\]
 In light of \eqref{GO4} and noting that
\[
\int_{0}^{T}e^{-2\rho^{3\over4}t}dt=\int_0^Te^{-2\rho^{3\over4}(T-t)}dt\leq \int_0^{+\infty}e^{-2\rho^{3\over4}t}dt={\rho^{-{3\over4}}\over2},
\]
we get by applying Cauchy-Schwarz's  inequality 
\bel{l7c}
\abs{\int_QZ(t,x)dxdt}\leq C\rho^{-{1\over4}}.
\ee
Since $u_{+,q}\in H_+$, 
\bel{l7c+}
\norm{g}_{H^{-{1\over4},-{1\over 2}}(\Sigma)}\leq C\norm{u_{+,q}}_{H_+}\leq C\norm{u_{+,q}}_{L^2(Q)}\leq e^{c\rho^2}
\ee
with $c=T+\sup_{x\in\overline{\Omega}}|x|$ and the same estimate holds for $w_{-,\widetilde{q}}$. Finally a combination of \eqref{l7b}, \eqref{l7c} and \eqref{l7c+} entails \eqref{l7a}.\end{proof}

\begin{proof}[Proof of Theorem \ref{tt1}] In this proof $c$ and $C$ are generic constants that can depend only on $m$, $Q$ and $s$. By Lemma \ref{l7}, 
\bel{tt1b-}
\abs{\widehat{p}(\zeta )}\leq C\left(\rho^{-{1\over4}}+\rho^{-1}R^{2}+\norm{\Lambda_{q}-\Lambda_{\widetilde{q}}}e^{c\rho^2}\right),\;\; |\zeta |\leq R ,\; \rho\ge\rho_0.
\ee
On the other hand
\begin{align*}
\| p \|_{H^{-1}(\R^{n+1})}^2&= \int_{\R^{n+1}} (1+\abs{\zeta}^2)^{-1}| \widehat{p}(\zeta)|^2  d\zeta 
\\
&=\int_{|\zeta ]\le R} (1+\abs{\zeta}^2)^{-1}| \widehat{p}(\zeta)|^2  d\zeta+\int_{|\zeta |\ge R} (1+\abs{\zeta}^2)^{-1}| \widehat{p}(\zeta)|^2  d\zeta
\\
&\le |\{|\zeta |\le R\}|\max_{|\zeta |\le R}| \widehat{p}(\zeta)|^2+R^{-2}\|\widehat{p}\|^2_{L^2}
\\
&\le |\{|\zeta |\le R\}|\max_{|\zeta |\le R}| \widehat{p}(\zeta)|^2+R^{-2}\|p\|_{L^2}^2
\\
&\le |\{|\zeta |\le R\}|\max_{|\zeta |\le R}| \widehat{p}(\zeta)|^2+R^{-2}m^2|Q|.
\end{align*}
Whence, as a consequence of  \eqref{tt1b-}, 
\bel{tt1b}
\| p \|_{H^{-1}(\R^{n+1})}^2\leq C\left(R^{-2}+\rho^{-{1\over2}}R^{n+1}+\rho^{-2}R^{n+5}+R^{n+1}\gamma ^2e^{c\rho^2}\right),\;\; R>0,\; \rho \ge \rho_0,
\ee
where we used the temporary notation $\gamma =\norm{\Lambda_{q}-\Lambda_{\widetilde{q}}}$. In this inequality, we take $R=\rho ^s$ in order to obtain, where $\alpha =1/2 -s(n+1) (>0)$,
\[
\| p \|_{H^{-1}(\R^{n+1})}^2\le C\left(\rho^{-\alpha}+\gamma ^2e^{c\rho ^2}\right),\rho \ge \rho _1.
\]
Here $\rho_1\ge \rho_0$ is constant depending only on $m$ and $Q$. A straightforward minimization argument with respect to $\rho$ yields
\begin{equation}\label{EQ1}
\| p \|_{H^{-1}(\R^{n+1})}\le C|\ln \gamma |^{-\alpha /4},
\end{equation}
if $\gamma \le \gamma ^\ast$, where $\gamma ^\ast$ is constant that can depend only on $m$, $Q$ and $s$.

\smallskip
When $\gamma \ge \gamma ^\ast$ we obtain, by using that $\| p \|_{H^{-1}(\R^{n+1})}\le C\|p\|_{L^2(\R^{n+1})}\le C|Q|^{1/2}m$,
\begin{equation}\label{EQ2}
\| p \|_{H^{-1}(\R^{n+1})}\le C\le C\frac{\gamma}{\gamma^\ast}.
\end{equation}
We complete the proof by noting that \eqref{tt1a} is obtained by combining \eqref{EQ1} and \eqref{EQ2}. 
 \end{proof}
 
%%%%%%%%%%%%%%%%%%%%%%%%%%%%%%%%

Before we proceed to the proof Theorem \ref{t1}, we recall a result quantifying the unique continuation of a real-analytic function from a measurable set.
\begin{Thm}
\label{thm-cpl}
$($\cite[Theorem 4]{AEWZ}$)$ Assume that $H:2\mathbb{B}\rightarrow \mathbb{C}$, where $\mathbb B$ the unit ball of $\R^{n+1}$, is real-analytic and satisfies 
$$
|\partial ^\alpha H(\kappa)|\leq K \frac{|\alpha |!}{\lambda  ^{|\alpha |}},\ \kappa \in 2\mathbb{B},\ \alpha \in \mathbb{N}^n,
$$
for some $(K,\lambda) \in \R_+^* \times (0,1]$.
Then for any measurable set $E \subset \mathbb{B}$ with positive Lebesgue measure, there exist two constants $M>0$ and $\theta \in (0,1)$, depending only on $\lambda$ and  $|E|$, so that
$$
\|H\|_{L^\infty (\mathbb{B})} \leq M K^{1-\theta} \left( \frac{1}{|E|}\int_E |H(\kappa)| d \kappa \right)^\theta.
$$
\end{Thm}

%%%%%%%%%%%%%%%%%%%%%%%%%%%%%%%%

\begin{proof}[Proof of Theorem \ref{t1}]
Let  $\delta>0$ be chosen in such a way that,  for any $\omega\in O_\delta=\{\eta\in\mathbb S^{n-1};\;|\eta-\omega_0|\leq\delta\} $, \[ \Gamma_{-,-\omega ,3\delta}\subset \mathcal{U}_-\;\;  \mbox{and}\;\; \Gamma_{-,\omega, 3\delta}\subset \mathcal{U}_+.\]  
In the rest of this proof, $C$ and $c$ denote generic constants that can depend only on $\Omega$, $Q$, $m$ and $\mathcal{U}_\pm$. Define 
$$
E_1:=\underset{\omega\in O_\delta}{\bigcup}\{(\xi,\tau )\in\R^n\times\R ;\; \xi\cdot\omega=0\},\quad E:=\{\zeta=(\xi,\tau)\in E_1;\; |\zeta |<1\}.
$$
In light of estimate \eqref{l7a},  for all $\zeta=(\xi,\tau)\in \{\eta\in E_1;\; |\eta|\leq R\}$, there exists $\omega\in O_\delta$ satisfying $\xi\cdot\omega=0$ such that 
$$ 
\abs{\widehat{p}(\zeta )}\leq C\left(\rho^{-{1\over4}}+\rho^{-1}R^{2}+\norm{\chi_{-,\omega ,\delta}(\Lambda_{q}-\Lambda_{\widetilde{q}})\chi_{+,\omega ,\delta}}e^{c\rho^2}\right),\;\; \rho \ge \rho_0.
$$
But, for any $\omega\in O_\delta$, $\mbox{supp}(\chi_{\pm,\omega ,\delta})\subset \mathcal{U}_\pm$. Whence
\[
\norm{\chi_{-,\omega,\delta}(\Lambda_{q}-\Lambda_{\widetilde{q}})\chi_{+,\omega,\delta}}\leq C\norm{\widehat{\Lambda}_{q}-\widehat{\Lambda}_{\widetilde{q}}}.
\]
Therefore
\bel{t1c}
\abs{\widehat{p}(\zeta )}\leq C\left(\rho^{-{1\over4}}+\rho^{-1}R^{2}+\norm{\widehat{\Lambda}_{q}-\widehat{\Lambda}_{\widetilde{q}}}e^{c\rho^2}\right),\;\; \zeta \in\{\eta\in E_1;\; |\eta|\leq R\},\; \rho \ge \rho_0.
\ee
Let
\[
H(\zeta )=\widehat{p}(R\zeta ).
\]
We repeat the same argument as in the proof of \cite[Theorem 1]{Ki2} in order to get
$$
\abs{\partial^\alpha H(\zeta )}\leq C\frac{e^{2R}|\alpha|!}{\lambda ^{\abs{\alpha}}},\;\;\zeta\in2\mathbb B,\;\;\alpha\in\mathbb N^{n+1},
$$
where $\lambda=[1+\max(T,\textrm{Diam}(\Omega))]^{-1}\in(0,1)$. An application of Theorem \ref{thm-cpl} with $K=e^{2R}$ yields
\[
\abs{\widehat{p}(R\zeta )}=\abs{ H(\zeta )}\leq \norm{H}_{L^\infty(\mathbb B)}\leq Ce^{2R(1-\theta)}\norm{H}_{L^\infty(E)}^\theta,\;\; \abs{\zeta}<1,
\]
for some $0<\theta<1$ depending on $Q$ and $\mathcal{U}_\pm$. But, from \eqref{t1c} we deduce
\[
\abs{\widehat{p}(R\zeta )}^2=\abs{ H(\zeta )}^2\leq C\left(\rho^{-{1\over2}}+\rho^{-2}R^{4}+\norm{\widehat{\Lambda}_{q}-\widehat{\Lambda}_{\widetilde{q}}}^2e^{c\rho^2}\right),\;\; \zeta\in E,\; \rho \ge \rho_0.
\]
Consequently, fixing $\sigma =4(1-\theta)$, we find
\begin{equation}\label{t1d}
\abs{\widehat{p}(\zeta )}^2\leq Ce^{\sigma R}\left(\rho^{-{1\over2}}+\rho^{-2}R^{4}+\norm{\widehat{\Lambda}_{q}-\widehat{\Lambda}_{\widetilde{q}}}^2e^{c\rho^2}\right)^\theta,\ \abs{\zeta}<R,\ \rho\ge \rho_0.
\end{equation}
We proceed as in the proof of Theorem \ref{tt1} in order to obtain
\[
\| p \|_{H^{-1}(\R^{n+1})}^{\frac{2}{\theta}}\leq C \left(R^{-2}+R^{n+1}e^{\sigma R}\left(\rho^{-{1\over2}}+\rho^{-2}R^{4}+\norm{\widehat{\Lambda}_{q}-\widehat{\Lambda}_{\widetilde{q}}}^2e^{c\rho^2}\right)^\theta\right)^{\frac{1}{\theta}},\;\; R>0,\; \rho \ge \rho_0.
\]
As $\mu \in (0,\infty ) \rightarrow \mu^{\frac{1}{\theta}}$ is convex, we derive

\[
\| p \|_{H^{-1}(\R^{n+1})}^{\frac{2}{\theta}}\le C\left(R^{-\frac{2}{\theta}}+\left[R^{n+1}e^{\sigma R}\right]^{\frac{1}{\theta}}\left(\rho^{-{1\over2}}+\rho^{-2}R^{4}+\norm{\widehat{\Lambda}_{q}-\widehat{\Lambda}_{\widetilde{q}}}^2e^{c\rho^2}\right)\right),\;\; R>0,\; \rho \ge \rho_0.
\]
Hence, there exists $\beta >0$ depending only on $n$ and $\theta$ so that
\[
\| p \|_{H^{-1}(\R^{n+1})}^{\frac{2}{\theta}}\le C\left(R^{-\frac{2}{\theta}}+e^{\beta R}\rho^{-{1\over2}}+e^{\beta R}\norm{\widehat{\Lambda}_{q}-\widehat{\Lambda}_{\widetilde{q}}}^2e^{c\rho^2}\right),\;\; R>0,\; \rho \ge \rho_0.
\]
Thus, there exists $R_0>0$ so that choosing $\rho$ satisfying $e^{\beta R}\rho^{-{1\over2}}=R^{-\frac{2}{\theta}}$ we find
\[
\|p\|_{H^{-1}(\R^{n+1})}^{\frac{2}{\theta}}\le C\left(R^{-\frac{2}{\theta}}+\norm{\widehat{\Lambda}_{q}-\widehat{\Lambda}_{\widetilde{q}}}^2e^{e^R}\right),\;\; R\ge R_0.
\]
The proof is completed similarly to that of Theorem \ref{tt1} by using a minimization argument.
\end{proof}

%%%%%%%%%%%%%%%%%%%%%%%%%%%%%%%%%%%%%%%%%%%%%%%%%%%%%%%%%%%%
\section{Determining the nonlinear term in a semi-linear IBVP from the DN map}

The objective in the actual section is the derivation of a stability estimate of the problem of determining the nonlinear term in a semi-linear parabolic IBVP from the corresponding ``linearized'' DtN map. We will give the precise definition of the ``linearized'' DtN map later in the text. The results of this section are obtained as a consequence of Theorem \ref{tt1}.

\smallskip
The linearization procedure we use require existence, uniqueness and a priori estimate of solutions of IBVP's under consideration. We preferred to work in the H\"older space setting for which we have a precise literature devoted to these aspects of solutions. However we are convinced that the same analysis can be achieved in the Sobolev space setting. But in that case this analysis seems to be more delicate.

\smallskip
In this section $\Omega$ is of class $C^{2+\alpha}$ for some  $0<\alpha <1$. The parabolic boundary of $Q$ is denoted by $\Sigma _p$. That is $\Sigma _p=\Sigma \cup\Omega_+$.

\smallskip
Consider the semilinear IBVP for the heat equation
\begin{equation}\label{1.1}
\left\{
\begin{array}{ll}
(\partial_t-\Delta) u+a(x,t,u)=0\quad &\mbox{in}\; Q,
\\
u=g &\mbox{on}\; \Sigma_p .
\end{array}
\right.
\end{equation}

We introduce some notations. We denote by $\mathscr{A}_0$ the set of functions from $C^1(\overline{Q}\times \mathbb{R})$ satisfying one of the following condition

\smallskip
(i) There exist two non negative constants $c_0$ and $c_1$ so that
\begin{equation}\label{c1}
ua(x,t,u)\ge -c_0u^2-c_1,\;\; (x,t,u)\in \overline{Q}\times \mathbb{R}.
\end{equation}

(ii) There exist a non negative constant $c_2$ and a non decreasing positive function of $\tau \ge 0$ satisfying
\[
\int_0^\infty \frac{d\tau}{\Phi (\tau )}=\infty
\]
so that
\begin{equation}\label{c1}
ua(x,t,u)\ge -|u|\Phi (|u|)-c_2,\;\; (x,t,u)\in \overline{Q}\times \mathbb{R}.
\end{equation}

Set $X=C^{2+\alpha ,1+\alpha /2}(\overline{Q})$ and let $X_0=\{g=G_{|\overline{\Sigma}_p};\; \mbox{for some}\; G\in X\}$. If $\|\cdot \|_X$ denotes  the natural norm on $X$ we equip $X_0$ with the quotient norm
\[
\|g\|_{X_0}=\inf\{\|G\|_X;\; G_{|\overline{\Sigma}_p}=g\}.
\]
By [Theorem 6.1, page 452, LSU], for any $a\in \mathscr{A}_0$ and $g\in X_0$, the IBVP \eqref{1.1} has a unique solution $u_{a,g}\in X$. Additionally, according to [Theorem 2.9, page 23, LSU], there exists a constant $C$ that can depend only on $Q$, $\mathscr{A}_0$ and $\underset{\overline{\Sigma}_p}{\max}|g|$ such that
\begin{equation}\label{1.2}
\max_{\overline{Q}}|u_{a,g}|\le C.
\end{equation}
A quick inspection of [inequalities (2.31) and (2.34), page 23, LSU] shows that \[\underset{\overline{\Sigma}_p}{\max}|g|\rightarrow C=C(\underset{\overline{\Sigma}_p}{\max}|g|)\] is non decreasing. 

\smallskip
Define the parabolic DtN map $N_a$ associated to $a\in \mathscr{A}_0$ by
\[
N_a:g\in X_0\longrightarrow \partial _\nu u_{a,g}\in Y=C^{1+\alpha ,(1+\alpha )/2}(\overline{\Sigma }).
\]

Note that, contrary to the preceding case, actually the DtN map $N_a$ is no longer linear. The linearization procedure consists then in computing the Fr\'echet derivative of $N_a$.

\smallskip
Let $\mathscr{A}$ be the subset of $\mathscr{A}_0$ of those functions $a$ satisfying  $\partial _ua\in C^2(\overline{Q}\times \mathbb{R})$. For $a\in \mathscr{A}$ and $h\in X_0$, consider the IBVP
\begin{equation*}
\left\{
\begin{array}{ll}
(\partial_t-\Delta) v+\partial _ua(x,t,u_{a,g}(x,t))v=0\quad &\mbox{in}\; Q,
\\
v=h &\mbox{on}\; \Sigma _p.
\end{array}
\right.
\end{equation*}
In light of [Theorem 5.4, page 322, LSU] the IBVP has a unique solution $v=v_{a,g,h}\in X$ satisfying
\[
\|v_{a,g,h}\|_X\le c\|h\|_{X_0}
\]
for some constant $c$ depending only on $Q$, $a$ and $g$. In particular $h\in X_0\rightarrow v_{a,g,h}\in X$ defines a bounded operator.

\begin{proposition}\label{propositionS1}
For each $a\in \mathscr{A}$, $N_a$ is continuously Fr\'echet differentiable and 
\[
N'_a(g) (h)=\partial _\nu v_{a,g,h}\in Y,\;\; g,h\in X_0.
\]
\end{proposition}

\begin{proof}
Let $a\in \mathscr{A}$. As $u\in X\rightarrow \partial _\nu u\in  Y$ is a bounded linear operator, it is enough to prove that $M_a:g\in X_0\rightarrow u_{a,g}\in X$ is continuously differentiable and $M'_a(g)(h)=v_{a,g,h}$, $g,h\in X_0$. To do that, we define $w\in X$ by
\[
w=u_{a,g+h}-u_g-v_{a,g,h}
\]
and set
\begin{align*}
&p(x,t)=\partial _ua(x,t,u_{a,g}(x,t)),
\\
&q(x,t)=\int_0^1(1-\tau)\partial_u^2 a(x,t,u_{a,g}+\tau (u_{a,g+h}-u_{a,g}))d\tau .
\end{align*}
According to Taylor's formula
\begin{align*}
a(x,t,u_{a,g+h}(x,t))-a(x,t,u_{a,g}(x,t))=p(x,t)(u_{a,g+h}&(x,t)-u_{a,g}(x,t))\\ &+q(x,t)(u_{a,g+h}(x,t)-u_{a,g}(x,t))^2.
\end{align*}
Consequently
\begin{align*}
a(x,t,u_{a,g+h}(x,t))-a(x,t,u_{a,g}(x,t))-p(x,t)v_{a,g,h}&(x,t)=p(x,t) w(x,t)\\ &+q(x,t)(u_{a,g+h}(x,t)-u_{a,g}(x,t))^2.
\end{align*}
Moreover, it is straightforward to check that $w$ is the solution of the IBVP
\begin{equation*}
\left\{
\begin{array}{ll}
(\partial_t-\Delta +p)w=-q(u_{a,g+h}-u_{a,g})^2\quad &\mbox{in}\; Q,
\\
w=0 &\mbox{on}\; \Sigma _p.
\end{array}
\right.
\end{equation*}
Set $Z=C^{\alpha ,\alpha /2}(\overline{Q})$. From inequality \eqref{1.2} and the comment following it, 
\[
\max_{\overline{Q}}|u_{a,g+h}|\le c,\;\; \mbox{for any}\; h\in B_{X_0}.
\]
Here and in the rest of this proof, $c$ is a generic constant can depend only on $Q$, $a$ and $g$. Whence
\[
\|p\|_Z,\; \|q\|_Z\le c,\;\; \mbox{for any}\; h\in B_{X_0},
\]
Therefore, again by [Theorem 5.4, page 322, LSU], it holds
\begin{equation}\label{1.3}
\|w\|_X\le c\|u_{a,g+h}-u_{a,g}\|_Z^2,\;\; \mbox{for any}\; h\in B_{X_0}.
\end{equation}
On the other hand $z=u_{a,g+h}-u_{a,g}$ is the solution of the IBVP
\begin{equation*}
\left\{
\begin{array}{ll}
(\partial_t-\Delta +r(x,t))z=0\quad &\mbox{in}\; Q,
\\
z=h &\mbox{on}\; \Sigma _p ,
\end{array}
\right.
\end{equation*}
with
\[
r(x,t)=\int_0^1\partial _u a(x,t,u_{a,g}+\tau (u_{a,g+h}-u_{a,g}))d\tau .
\]
Proceeding as above, we get
\[
\|z\|_X\le c\|h\|_{X_0}.
\]
This estimate, combined with \eqref{1.3}, yields
\[
\|w\|_X\le c \|h\|_{X_0}^2.
\]
That is we proved that $M_a$ is differentiable at $g$ and $M'_a(g)(h)=v_{a,g,h}$, $h\in X_0$.

\smallskip
It remains to establish the continuity of $g\in X_0\rightarrow M_a'(g)\in \mathscr{B}(X_0,X)$. To this end, let $g,k,h\in X_0$ and set
\[
\varphi =v_{a,g+k,h}-v_{a,g,h}.
\]
We see that $\varphi$ is the unique solution of the IBVP
\begin{equation*}
\left\{
\begin{array}{ll}
\left[\partial _t -\Delta  +\partial _ua(x,t,u_{a,g+k}(x,t))\right]\varphi =\alpha (x,t)(u_{a,g}-u_{a,g+k})v_{a,g,h}\quad &\mbox{in}\; Q,
\\
\varphi =0 &\mbox{on}\; \Sigma _p ,
\end{array}
\right.
\end{equation*}
where
\[
\alpha (x,t)=\int_0^1\partial _u^2a(x,t, u_{a,g}(x,t)+\tau (u_{a,g+k}(x,t)-u_{a,g}(x,t)))d\tau .
\]
Assume that $k,h\in B_{X_0}$. By proceeding one more time as above we get
\[
\|\varphi \|_X\le c\|u_{a,g+k}-u_{a,g}\|_Z\|v_{a,g,h}\|_Z\le c\|u_{a,g+k}-u_{a,g}\|_Z.
\]
Whence
\[
\|M_a'(g+k)-M_a'(g)\|_{\mathscr{B}(X_0,X)}\le c\|M_a(g+k)-M_a(g)\|_X,
\]
which leads immediately to the continuity of $M_a'$ since $M_a$ is continuous.
\end{proof}

In order to handle the inverse problem corresponding to the semi-linear IBVP \eqref{1.1} we need to extend the operator $\Lambda _q$ by varying also the initial condition. To do that we start by considering the IBVP
\begin{equation}\label{1.4}
\left\{
\begin{array}{ll}(\partial_t-\Delta +q(t,x))u=0\;\; \textrm{in}\; Q,
\\  
u_{|\Omega _+}=u_0,
\\ 
u_{|\Sigma }=g.
\end{array}
\right.
\end{equation}

Let $X_+= r_+H_+\subset H^{-1}(\Omega )$ that we equip with its natural quotient norm
\[
\|u_0\|_{X_+}=\inf\{\|u\|_{H_+};\; r_+u=u_0\}.
\]
Let $(u_0,g)\in X_+\oplus  H^{-{1\over 2},-{1\over4}}(\Sigma)$ and pick $v\in H_+$ so that $r_+v=u_0$. Formally, if $u$ is a solution of the IBVP \eqref{1.4} then $w=u-v$ is the solution of the IBVP
\begin{equation}\label{1.4}
\left\{
\begin{array}{ll}(\partial_t-\Delta +q(t,x))w=f\;\; \textrm{in}\; Q,
\\  
w_{|\Omega _+}(0,\cdot)=0,
\\ 
w_{|\Sigma }=h=g-\tau_0v \in H^{-{1\over 2},-{1\over4}}(\Sigma).
\end{array}
\right.
\end{equation}
Here
\[
f=-\partial_tv+\Delta v-q(t,x)v\in L^2(Q).
\]
Fix $q\in mB_{L^\infty (Q)}$. When $g=0$ (resp. $f=0$) the IBVP problem \eqref{1.4} has a unique solution $w^1_{q,f}\in H^{2,1}(Q)$ \cite[Theorem 1.43, page 27]{Ch} (resp. $w^2_{q,h}\in H_+$ by Proposition \ref{p2})
and
\begin{align*}
&\|w^1_{q,f}\|_{H^{2,1}(Q)}\le C_0\| f\|_{L^2(Q)}\le C_1 \|v\|_{H_+},
\\
& \|w^2_{q,h}\|_{H_+}\le C_2\|g-\tau_0v\|_{H^{-{1\over 2},-{1\over4}}(\Sigma)}\le C_3\left(\|g\|_{H^{-{1\over 2},-{1\over4}}(\Sigma)}+\|v\|_{H_+}\right),
\end{align*}
for some constants $C_i$ depending only on $Q$ and $m$. Whence, $u_{q,u_0,g}=v+ w^1_{q,f}+w^2_{q,h}\in H_+$ is the unique solution of the IBVP \eqref{1.4} and there exists a constant $C>0$ that can depend only on $Q$ and $m$ so that
\[
\|u_{q,u_0,g}\|_{H_+}\le C\left( \|g\|_{H^{-{1\over 2},-{1\over4}}(\Sigma)}+\|v\|_{H_+}\right).
\]
Since $v\in H_+$ is chosen arbitrary so that $r_+v=u_0$, we derive
\begin{equation}\label{1.5}
\|u_{q,u_0,g}\|_{H_+}\le C\left( \|g\|_{H^{-{1\over 2},-{1\over4}}(\Sigma)}+\|u_0\|_{X_+}\right).
\end{equation}
Therefore, according to the trace theorem in Proposition \ref{p1}, the extended parabolic DtN map
\begin{align*}
\Lambda _q^e:X_+\oplus  H^{-{1\over 2},-{1\over4}}(\Sigma) &\rightarrow H^{-{3\over 2},-{3\over4}}(\Sigma)
\\
(u_0,g)&\mapsto \tau_1u_{q,u_0,g}
\end{align*}
defines a bounded operator. 

\smallskip
We need a variant of Theorem \ref{t1}. To this end, we recall that from \cite[Lemma 12.3, page 73]{LM1} we have the following interpolation inequality, where $c_0$ is a constant depending only on $Q$,
\begin{equation}\label{1.6}
\|u\|_{L^2(Q)}\le c_0 \|u\|_{H^1(Q)}^{1/2}\|u\|_{H^{-1}(Q)}^{1/2},\;\; u\in H^1(Q).
\end{equation}
On the other hand \begin{equation}\label{1.7}
\|u\|_{C(\overline{Q})}\le c_1 \|u\|_{C^1(\overline{Q})}^{\frac{n+1}{n+3}}\|u\|_{L^2(Q)}^{\frac{2}{n+3}},\;\; u\in C^1(\overline{Q}),
\end{equation}
where $c_1$ is a constant depending only on $Q$. 

\smallskip
This interpolation inequality is more or less known but, for sake of completeness, we provide its proof in Lemma \ref{ii} of Appendix B.

\smallskip
We get by combining these two interpolation inequalities the following one
\begin{equation}\label{1.8}
\|u\|_{C(\overline{Q})}\le c \|u\|_{C^1(\overline{Q})}^{\frac{n+2}{n+3}}\|u\|_{H^{-1}(Q)}^{\frac{1}{n+3}},\;\; u\in C^1(\overline{Q}),
\end{equation}
for some constant $c$ depending only on $Q$. 

\medskip
If $\frac{1}{2(n+3)}<s<\frac{1}{2(n+1)}$,
\begin{equation}\label{Psi}
\Theta _s(\rho )= |\ln \rho |^{-\frac{1-2s(n+1)}{n+3}}+\rho ,\;\; \rho >0,
\end{equation}
extended by continuity at $\rho =0$ by setting $\Theta _s(0)=0$. 

\smallskip
Inspecting the proof of Theorem \ref{tt1}, using the interpolation inequality \eqref{1.8} and that 
\[
\norm{\Lambda_q-\Lambda_{\widetilde{q}}}\le \norm{\Lambda^e_q-\Lambda^e_{\widetilde{q}}},
\]
we get 

\begin{Thm}\label{t_s1} 
Fix $m>0$ and $\frac{1}{2(n+3)}<s<\frac{1}{2(n+1)}$. There exists a constant $C>0$, that can depend only on $m$, $Q$ and $s$, so that for any  $q,\widetilde{q} \in mB_{C^1(\overline{Q})}$,
\[
\norm{q-\widetilde{q} }_{C(\overline{Q})}\leq C\Theta _s\left(\norm{\Lambda^e_q-\Lambda^e_{\widetilde{q}}}\right).
\]
Here $\norm{\Lambda^e_q-\Lambda^e_{\widetilde{q}}}$ stands for the norm of $\Lambda^e_q-\Lambda^e_{\widetilde{q}}$ in $\mathscr{B}(X_+\oplus H^{-{1\over2},-{1\over4}}(\Sigma);H^{{1\over2},{1\over4}}(\Sigma))$. 
\end{Thm}
Fix $\lambda >0$. From \eqref{1.2} and the remark following it, there exists a constant $c_\lambda >0$ so that
\[
\max_{\overline{Q}} |u_{a,g}| \le c_\lambda ,\;\; a\in \mathscr{A}_0,\; \max_{\overline{\Sigma} _p}|g|\le \lambda .
\]
For fixed $\delta >0$, consider
\[
\widehat{\mathscr{A}}=\{a=a(x,u)\in \mathscr{A};\; \|\partial _ua\|_{C(\overline{\Omega}\times [-c_\lambda ,c_\lambda ])}\le \delta\}.
\]
To $a\in \widehat{\mathscr{A}}$ and $g\in X_0$ we associate 
\[
p_{a,g}(x,t)=\partial _ua(x,u_{a,g}(x,t)),\;\; (x,t)\in \overline{Q}.
\]
It is straightforward to check that
\[
N'_a(g)=\Lambda^e_{p_{a,g}}{_{|X_0}}.
\]
From now on $N'_a(g)-N'_{\widetilde{a}}(g)$ is considered as a bounded operator from $X_0$ endowed with norm of $X_+\oplus H^{-{1\over2},-{1\over4}}(\Sigma)$ into $H^{{1\over2},{1\over4}}(\Sigma)$.

\smallskip
Since $\|p_{a,g}\|_{L^\infty(Q)}\le \delta$ for any $a\in \widehat{\mathscr{A}}$ and $g\in X_0$ so that $\max_{\overline{\Sigma} _p}|g|\le \lambda$, we get as a consequence of Proposition \ref{p2}
\[
\sup\{\|N'_a(g)-N'_{\widetilde{a}}(g)\|;\; a\in \widehat{\mathscr{A}},\; g\in X_0\; \mbox{and}\; \max_{\overline{\Sigma} _p}|g|\le \lambda\} <\infty .
\]
Bearing in mind that $C^\infty (\overline{Q})$ is dense in $H_+$, we derive that $X_0$ is dense $X_+\oplus H^{-{1\over2},-{1\over4}}(\Sigma)$. Thus, 
\begin{equation}\label{1.9}
\| N'_a(g)-N'_{\widetilde{a}}(g)\|=\|\Lambda^e_{p_{a,g}}-\Lambda^e_{p_{\widetilde{a},g}}\|.
\end{equation}
Pick $a_0\in C^1(\overline{\Omega })$ and set
\[
\widehat{\mathscr{A}}_0=\{a\in \widehat{\mathscr{A}} ;\; a(\cdot ,0)=a_0\}.
\]
We note that when $g\equiv s$, $|s|\le \lambda$, we have
\[
p_{a,g}(x,0)=\partial _ua(x,u_{a,g}(x,0))=\partial _ua(x,s),\;\; x\in \overline{\Omega}.
\]
In light of this identity and \eqref{1.9} we obtain as a consequence of Theorem \ref{t_s1}
\begin{Thm}\label{t_s2} 
Fix $\lambda>0$ and $\frac{1}{2(n+3)}<s<\frac{1}{2(n+1)}$. There exists a constant $C>0$, that can depend only on $\lambda $, $s$, $Q$ and $\widehat{\mathscr{A}}_0$, so that for any  $a,\widetilde{a} \in \widehat{\mathscr{A}}_0$,
\[
\norm{a-\widetilde{a} }_{C(\overline{\Omega}\times [-\lambda ,\lambda ])}\leq C\Theta_s \left(\sup_{g\in X_{0,\lambda}}\norm{N'_a(g)-N'_{\widetilde{a}}(g)}\right).
\]
Here $X_{0,\lambda}=\{g\in X_0;\; \max_{\overline{\Sigma} _p}|g|\le \lambda\}$ and  $\norm{N'_a(g)-N'_{\widetilde{a}}(g)}$ stands for the norm of $N'_a(g)-N'_{\widetilde{a}}(g)$ in $\mathscr{B}(X_+\oplus H^{-{1\over2},-{1\over4}}(\Sigma);H^{{1\over2},{1\over4}}(\Sigma))$. 
\end{Thm}
We now turn our attention to the special case $a=a(u)$ for which we are going to show that we have a stability estimate with less data than in the case $a=a(x,u)$.

\smallskip
Define
\[
\mathcal{A}=\{ a\in C^3(\mathbb{R});\; a(0)=0\; \mbox{and $a$ is positive increasing}\}.
\]
It is straightforward to check that $\mathcal{A}\subset \mathscr{A}$. Let
\[
Y_0=C_{,0}^{2+\alpha ,1+\alpha /2}(\overline{\Sigma })=\{g\in C^{2+\alpha ,1+\alpha /2}(\overline{\Sigma });\; g(\cdot ,0)=0\}
\]
that we identify to the subset of $X_0$ given by $\{g\in X_0;\; g_{|\Omega _+}=0\}$. We denote again the solution of \eqref{1.1} by $u_{a,g}$ when $a\in \mathcal{A}$ and $g\in Y_0$. In that case $N_a$ is considered as a map from $Y_0$ into $Y$. Fix $g\in Y_0$ and let $a\in \mathcal{A}$. Since $a(0)=0$, setting $q(t,x)=\int_0^1a'(\tau u(x,t))d\tau$, we easily see that $u_{a,g}$ is also the solution of the IBVP
\begin{equation*}
\left\{
\begin{array}{ll}
(\partial_t-\Delta +q(x,t))u=0\quad &\mbox{in}\; Q,
\\
u=g &\mbox{on}\; \Sigma _p.
\end{array}
\right.
\end{equation*}
But $q\ge 0$. Whence
\[
\min_{\overline{Q}}u=-\max_{\overline{\Sigma}}g^-=a(g),\quad \max_{\overline{Q}}u=\max_{\overline{\Sigma}}g^+=b(g)
\]
according to the weak maximum principle (see for instance \cite[Theorem 4.25, page 121]{RR}). In particular
\[
u(\overline{Q})=I_g=[a(g),b(g)].
\]
We derive by mimicking the analysis before Theorem \ref{t_s2} 

\begin{Thm}\label{t_s2} 
Fix $g\in Y_0$ non constant and $\frac{1}{2(n+3)}<s<\frac{1}{2(n+1)}$. There exists a constant $C>0$, that can depend only on $g$, $s$, $Q$ and $\mathcal{A}$, so that for any  $a,\widetilde{a} \in \mathcal{A}$,
\[
\norm{a-\widetilde{a} }_{C(I_g)}\leq C\Theta_s \left(\norm{N'_a(g)-N'_{\widetilde{a}}(g)}\right).
\]
Here $\norm{N'_a(g)-N'_{\widetilde{a}}(g)}$ denotes the norm of $N'_a(g)-N'_{\widetilde{a}}(g)$ in $\mathscr{B}(H^{-{1\over2},-{1\over4}}(\Sigma);H^{{1\over2},{1\over4}}(\Sigma))$. 
\end{Thm}

\begin{rem} 
{\rm
Other stability results can be obtained in a similar manner to that we used in the present section. We just mention one of them. To this end, let $\widehat{\mathscr{A}}_0$ be defined as before with the only difference that we actually permit to functions of $\widehat{\mathscr{A}}_0$ to depend also on the time variable $t$. 

\smallskip
Let $a, \widetilde{a}\in \widehat{\mathscr{A}}_0$ and pick $(x_0,t_0,u_0)\in \Gamma \times (0,T)\times [-\lambda ,\lambda ]$ so that
\begin{equation}\label{rem6.1}
|(a-\widetilde{a})(x_0,t_0,u_0)|=\frac{1}{2}\|a-\widetilde{a}\|_{C(\Gamma \times [0,T]\times [-\lambda ,\lambda ])}.
\end{equation}
Let $\epsilon =\min (t_0,T-t_0)$ and $g\in X_{0,\lambda}$ so that $g=s$ on $\Gamma \times [\epsilon ,T-\epsilon]$ for some $|s|\le \lambda$. We proceed as in the proof of Theorem \ref{t_s2} in order to derive
\[
\norm{a-\widetilde{a} }_{C(\Gamma \times[\epsilon, T-\epsilon]\times [-\lambda ,\lambda ])}\le C\Theta_s \left(\sup_{g\in X_{0,\lambda}}\norm{N'_a(g)-N'_{\widetilde{a}}(g)}\right),
\]
where the constant $C$ depends only on $\lambda $, $s$, $Q$ and $\widehat{\mathscr{A}}_0$.

\smallskip
In light of \eqref{rem6.1} this estimate yields 
\[
\norm{a-\widetilde{a} }_{C(\Gamma \times[0, T]\times [-\lambda ,\lambda ])}\le C\Theta_s \left(\sup_{g\in X_{0,\lambda}}\norm{N'_a(g)-N'_{\widetilde{a}}(g)}\right).
\]
 }
\end{rem}

%%%%%%%%%%%%%%%%%%%%%%%%%%%%%%%%%%%%%%%%%%%%%%%%%%%%%
\appendix

\section{}

\begin{lem}\label{iii} 
$H^{s}(0,T;L^2(\Gamma))=H^{s}_0(0,T;L^2(\Gamma))$ for any $0\le s\le 1/2$.
\end{lem}

\begin{proof} 
Let $\mathcal{D}_0$ (resp. $\mathcal{D}_T$) be the subspace of $C_0^\infty (\mathbb{R}, L^2(\Gamma ))$ of functions vanishing in a neighborhood of $t=0$ (resp. $t=T$). By \cite[Lemma 11.1, page 55]{LM1} both $\mathcal{D}_0$ and $\mathcal{D}_T$ are dense in $H^{1/2}(\mathbb{R}, L^2(\Gamma ))$.

\smallskip
Fix $u\in \mathcal{D}_0$ and let $\epsilon >0$ so that $u=0$ in $(-2\epsilon ,2\epsilon )\times \Gamma$. As  $\mathcal{D}_T$ is dense in $H^{1/2}(\mathbb{R}, L^2(\Gamma ))$, there exists a sequence $(u_n)$ in $\mathcal{D}_T$ that converges to $u$ in $H^{1/2}(\mathbb{R}, L^2(\Gamma ))$. Pick $\psi\in C_0^\infty(-2\epsilon,2\epsilon)$ satisfying $\psi=1$ on $(-\epsilon,\epsilon)$.

\smallskip
We get by taking into account that $\psi u_n =\psi (u_n-u)$
$$
\norm{\psi u_n}_{H^{{1\over2}}(\R;L^2(\Gamma))}=\norm{\psi (u_n-u)}_{H^{{1\over2}}(\R;L^2(\Gamma))}\leq C\norm{u_n-u}_{H^{{1\over2}}(\R;L^2(\Gamma))},
$$
for some constant depending only on $\psi$. Whence $((1-\psi)u_n)$ is a sequence in $\mathcal{D}_0\cap \mathcal{D}_T$ converging to $u$ in $H^{1/2}(\mathbb{R}, L^2(\Gamma ))$.

\smallskip
On the other hand, we know that $H^{1/2}(0,T;L^2(\Gamma ))$ can be seen as the (quotient) space of the restriction to $(0,T)$ of functions from $H^{1/2}(\mathbb{R}, L^2(\Gamma ))$. Therefore any function from $H^{1/2}(0,T;L^2(\Gamma ))$ can be approximated, with respect to the norm of $H^{1/2}(0,T;L^2(\Gamma ))$, by a sequence of functions from $\mathcal{D}_0\cap \mathcal{D}_T$, that is to say by a sequence from $C_0^\infty ((0,T);L^2(\Gamma ))$.

\smallskip 
The proof follows by noting that $H^{{1\over2}}(0,T;L^2(\Gamma))$  is continuously and densely embedded in $H^{s}(0,T;L^2(\Gamma))$, $0\le s\le 1/2$.
\end{proof}

\section{}

\begin{lem}\label{ii}
Let $D$ be a bounded domain of $ \mathbb{R}^d$, $d\ge 2$, of class $C^{0,\alpha}$ with $0<\alpha \le 1$. There exists a constant $C>0$ depending only on $D$ and $\alpha$ so that, for any $u\in C^{0,\alpha}(\overline{D} )$,
\[
\|u\|_{C(\overline{D})}\le C\|u\|_{C^{0,\alpha}(\overline{D})}^{\frac{d}{d+2\alpha}}\|u\|_{L^2(D)}^{\frac{2\alpha }{d+2\alpha}}.
\]
\end{lem}

\begin{proof} 
Let $u\in C^{0,\alpha}(\overline{D})$. From \cite[Lemma 6.37, page 136. In fact this lemma is stated with $C^{k,\alpha}$-regularity, $k\ge 1$, but a careful inspection of the proof shows that this lemma can be extended to the case of $C^{0,\alpha}$-regularity]{GT} and its proof, there exists $v\in C^{0,\alpha}(\mathbb{R}^d)$ with compact support so that $v=u$ in $\overline{D}$ and 
\begin{equation}\label{a1}
\|v\|_{C^{0,\alpha}(\mathbb{R}^d)}\le \kappa \|u\|_{C^{0,\alpha}(\overline{D})},\quad \|v\|_{L^2(\mathbb{R}^d)}\le \kappa \|u\|_{L^2(D)},
\end{equation}
where the constant $\kappa$ depends only on $D$.

\smallskip
For $x\in \overline{D}$ and $r>0$, we get in a straightforward manner 
\begin{align*}
|u(x)||B(x,r)|&\le \int_{B(x,r)}|v(x)-v(y)|dy + \int_{B(x,r)}|v(y)|dy
\\
&\le \|v\|_{C^{0,\alpha}(\mathbb{R}^d)}\int_{B(x,r)}|x-y|^\alpha dy+|B(x,r)|^{1/2}\|v\|_{L^2(\mathbb{R}^n)}.
\end{align*}
This estimate combined with \eqref{a1} yields
\begin{equation}\label{a2}
|u(x)|\le C\left(r^\alpha \|u\|_{C^{0,\alpha}(\overline{D})}+r^{-d/2}\|u\|_{L^2(D)}\right),
\end{equation}
where the constant $C$ depends only on $D$ and $\alpha$. The expected inequality follows by taking $r$ in \eqref{a2} so that $r^\alpha \|u\|_{C^{0,\alpha}(\overline{D})}=r^{-d/2}\|u\|_{L^2(D)}$.
\end{proof}

%%%%%%%%%%%%%%%%%%%%%%%%%%%%%%%%%%%%%%%%%%%%%%%%%%%%%

%


\begin{thebibliography}{99}
%



\bibitem{A} {\sc G. Alessandrini}, {\em Stable determination of conductivity by boundary measurements}, Appl. Anal. \textbf{27} (1988), 153-172. 
\bibitem{AEWZ}{\sc J. Apraiz, L. Escauriaza, G. Wang, C. Zhang}, {\em Observability inequalities and measurable sets}, J. Eur. Math. Soc. \textbf{16} (2014), 2433-2475.

 \bibitem{Be} {\sc I. Ben Aicha}, {\em  Stability estimate for hyperbolic inverse problem with time dependent coefficient}, Inverse Problems \textbf{31} (2015), 125010.
\bibitem{BJY}{\sc M. Bellassoued, D. Jellali, M. Yamamoto}, {\em  Lipschitz stability for a hyperbolic inverse problem by finite local boundary data}, Appl. Anal. \textbf{85} (2006),  1219-1243.
\bibitem{BK}{\sc  A. Bukhgeim and M. Klibanov}, {\em Global uniqueness of a class of multidimensional inverse problem},
Sov. Math. Dokl.
\textbf{24}
(1981), 244-247.
\bibitem{BU}{\sc A. L. Bukhgeim and G. Uhlmann}, {\em Recovering a potential from partial  Cauchy data}, Commun. Partial Diff. Equat. \textbf{27} (3-4) (2002),  653-668.
\bibitem{CE86-1}{\sc J. R. Cannon and S. P. Esteva}, {\em An inverse problem for the heat equation}, Inverse Problems \textbf{2} (1986), 395-403.
\bibitem{CE86-2}{\sc J. R. Cannon and S. P. Esteva}, {\em A note on an inverse problem related to the 3-D heat equation},
Inverse problems (Oberwolfach, 1986), 133-137, Internat. Schriftenreihe Numer. Math. 77,
Birkh\"auser, Basel, 1986.
\bibitem{CL88}{\sc 	J. R. Cannon and Y. Lin}, {\em Determination of a parameter p(t) in some quasi-linear parabolic differential equations}, Inverse Problems  \textbf{4} (1988), 35-45.
\bibitem{CL90}{\sc 	J. R. Cannon and Y. Lin}, {\em An Inverse Problem of Finding a Parameter in a Semi-linear Heat Equation},  J. Math. Anal. Appl. \textbf{145} (1990), 470-484.
\bibitem{CDR} {\sc P. Caro, D. Dos Santos Ferreira and A. Ruiz}, {\em Stability estimates for the Radon transform with restricted data and applications}, J. Diff. Equat. \textbf{260} (2016), 2457-2489.
\bibitem{Ch91-1}{\sc M. Choulli}, {\em An abstract inverse problem}, J. Appl. Math. Stoc. Ana. \textbf{4} (2) (1991) 117-128.
\bibitem{Ch91-2}{\sc M. Choulli}, {\em An abstract inverse problem and application}, J. Math. Anal. Appl. \textbf{160} (1) (1991), 190-202.
 \bibitem{Ch}{\sc M. Choulli}, {\em Une introduction aux probl\`emes inverses elliptiques et paraboliques}, Math\'ematiques et Applications, Vol. 65, Springer-Verlag, Berlin, 2009.
%
\bibitem{CK}{\sc M. Choulli and Y. Kian}, {\em Stability of the determination of a time-dependent coefficient in parabolic equations}, Math. Control \& Related fields {\bf 3} (2) (2013), 143-160.
\bibitem{CKS1}{\sc M. Choulli, Y. Kian, E. Soccorsi},  {\em Determining the time dependent external potential from the DN map in a periodic quantum waveguide},   SIAM J. Math.  Anal. \textbf{47} (6) (2015), 4536-4558.
 \bibitem{CKS2} {\sc M. Choulli, Y. Kian, E. Soccorsi}, {\em Double logarithmic stability estimate in the identification of a scalar potential by a partial elliptic Dirichlet-to-Neumann map},  Bulletin SUSU MMCS \textbf{8} (3) (2015), 78-94.
 \bibitem{CKS3} {\sc M. Choulli, Y. Kian, E. Soccorsi}, {\em  Stable Determination of Time-Dependent Scalar Potential From Boundary Measurements in a Periodic Quantum Waveguide}, New Prospects in Direct, Inverse and Control Problems for Evolution Equations, A. Favini, G. Fragnelli and R. M. Mininni (Eds), Springer-INdAM, Roma, 2014, 93-105.
 \bibitem{CKS4} {\sc M. Choulli, Y. Kian, E. Soccorsi}, {\em  Stability result for elliptic inverse periodic coefficient problem by partial Dirichlet-to-Neumann map},  arXiv:1601.05355.
 \bibitem{CKS5} {\sc M. Choulli, Y. Kian, E. Soccorsi}, {\em On the Calder\'on problem in periodic cylindrical domain with partial Dirichlet and Neumann data}, arXiv:1601.05358.
 \bibitem{COY} {\sc M. Choulli, E. M. Ouhabaz, M. Yamamoto}, {\em Stable determination of
a semilinear term in a parabolic equation}, Commun. Pure Appl. Anal. \textbf{5} (3)
(2006), 447-462.
\bibitem{CY06}{\sc M. Choulli and  M. Yamamoto},
{\em Some stability estimates in determining sources and coefficients},
J. Inv. Ill-Posed Problems \textbf{14} (4) (2006), 355-373.
\bibitem{CY11}{\sc M. Choulli and  M. Yamamoto}, {\em Global existence and stability for an inverse coefficient problem for a semilinear parabolic equation}, Arch. Math (Basel)  \textbf{97} (6) (2011), 587-597.
 \bibitem{FK}{\sc K. Fujishiro and Y. Kian}, {\em Determination of time dependent factors of coefficients in fractional diffusion equations}, Math. Control \& Related fields \textbf{6} (2016), 251-269.
\bibitem{GK}{\sc P. Gaitan and Y. Kian}, {\em A stability result for a time-dependent potential in a cylindrical domain}, Inverse Problems \textbf{29} (6) (2013), 065006.
 \bibitem{GT} {\sc D. Gilbarg} and {\sc N. S. Trudinger}, {\em Elliptic partial differential equations of second order}, 2nd ed., Springer-Verlag, Berlin, 1983.
\bibitem{HW} {\sc H. Heck and J.-N. Wang}, {\em Stability estimate for the inverse boundary value problem by partial Cauchy data}, Inverse Problems
\textbf{22} (2006), 1787-1797
\bibitem{I}{\sc V. Isakov}, {\em Completness of products of solutions and some inverse problems for PDE}, J. Diff. Equat. \textbf{92} (1991), 305-316.
\bibitem{I2}{\sc  V. Isakov},  {\em On uniqueness in inverse problems for semilinear parabolic equations}, Arch. Rat. Mech. Anal. \textbf{124} (1993), 1-12.
\bibitem{I3}{\sc  V. Isakov},  {\em Uniqueness of recovery of some systems of semilinear partial differential equations}, Inverse Problems \textbf{17} (2001), 607-618.
\bibitem{I4}{\sc  V. Isakov},  {\em  Uniqueness and Stability in Inverse Parabolic Problems}, Proc. of GAMM-SIAM Conference on "Inverse Problems in Diffusion Processes",  (1994), St.Wolfgang, Oesterreich. SIAM, Philadelphia, (1995).
\bibitem{KSU} {\sc C.E. Kenig, J. Sj\"ostrand, G. Uhlmann}, {\em The Calderon problem with partial data}, Ann. of Math. {\bf 165} (2007), 567-591.

 \bibitem{Ki1}{\sc Y. Kian}, {\em Unique determination of a time-dependent potential for  wave equations from partial data}, arXiv:1505.06498.
 \bibitem{Ki2}{\sc Y. Kian}, {\em Stability in the determination of a time-dependent coefficient for wave equations from partial data}, J. Math. Anal. and Appl. \textbf{436} (2016), 408-428.
 \bibitem{Ki3}{\sc Y. Kian}, {\em Recovery of time-dependent damping coefficients and potentials appearing in wave equations from partial data}, arXiv:1603.09600.
 \bibitem{Kl}{\sc M. Klibanov}, {\em Global uniqueness of a multidimensional inverse problem for a nonlinear parabolic equation by a Carleman estimate}, Inverse Problems \textbf{20} (2004), 1003.
\bibitem{LM1}{\sc J-L. Lions and E. Magenes}, {\em Non homogeneous boundary value problems and applications}, Volume I, Springer Verlag, Berlin, 1972.
\bibitem{LM2}{\sc J-L. Lions and E. Magenes}, {\em Non homogeneous boundary value problems and applications}, Volume II, Springer Verlag, Berlin, 1972.
\bibitem{RR} {\sc M. Renardy} and {\sc R. C. Rogers}, {\em An introduction to partial differential equations}, Springer Verlag, NY, 1992.
\bibitem{SU} {\sc J. Sylvester and G. Uhlmann}, {\em A global uniqueness theorem for an inverse boundary value problem}, Ann. of Math. \textbf{125} (1987), 153-169.

\end{thebibliography}
\end{document}